\documentclass{amsart}

\usepackage{amsthm,amsmath,amsthm,amssymb,mathrsfs,graphicx,mathtools,relsize,hyperref,cleveref,enumerate,nameref}
\usepackage[utf8]{inputenc}
\usepackage[all]{xy}

\theoremstyle{theorem}
 \newtheorem{thm}{Theorem}[section]
 \newtheorem{prop}[thm]{Proposition}
 \newtheorem{lem}[thm]{Lemma}
 \newtheorem{cor}[thm]{Corollary}

 \newtheorem{que}[thm]{Question}

\theoremstyle{definition}
 \newtheorem{exm}[thm]{Example}
 \newtheorem{dfn}[thm]{Definition}
 \newtheorem{notation}[thm]{Notation}
\theoremstyle{remark}
 
 \numberwithin{equation}{section}
 
\renewcommand{\le}{\leqslant}
\renewcommand{\ge}{\geqslant}
\renewcommand{\setminus}{\smallsetminus}
\def\Aut{\text{\rm Aut}}

\def\Prim{\text{\rm Prim}}
\def\Max{\text{\rm Max}}
\def\id{\text{\rm id}}
\def\Ad{\text{\rm Ad}}

\def\hull{\text{\rm hull}}

\def\supp{\text{\rm supp}}
\def\ker{\text{\rm ker}}
\def\hull{\text{\rm hull}}

\def\tor{\text{\sf tor}}

\def\H{\mathcal{H}}

\def\B{\mathcal{B}}

\def\BH{\mathcal{B}(\mathcal{H})}

\def\supp{\text{\rm supp}}
\def\Td{T_{d}}

\newcommand{\norm}[1]{\left\lVert#1\right\rVert}

\newcommand{\hlight}[1]{\textit{\textbf{#1}}}

\title[Weighted Orlicz $*$-algebras on locally elliptic groups]{Weighted Orlicz $*$-algebras on locally elliptic groups}
\date{\today}

\thanks{The author acknowledges support from the FWO and F.R.S.-FNRS under the Excellence of Science (EOS) program (project ID 40007542).}

\author[Max Carter]{Max Carter} 
\address{Institut de recherche en mathématique et physique \\
Chemin du Cyclotron 2 \\
boîte L7.01.02 \\
Université catholique de Louvain \\ 
1348 Louvain-la-Neuve \\
Belgique.}
\email{max.carter@uclouvain.be}

\begin{document}

\begin{abstract}
Let $G$ be a locally elliptic group, $(\Phi,\Psi)$ a complementary pair of Young functions, and $\omega: G \rightarrow [1,\infty)$ a weight function on $G$ such that the weighted Orlicz space $L^\Phi(G,\omega)$ is a Banach $*$-algebra when equipped with the convolution product and involution $f^*(x):=\overline{f(x^{-1})}$ ($f \in L^\Phi(G,\omega)$).  Such a weight always exists on $G$ and we call it an \textit{$L^\Phi$-weight}. We assume that $1/\omega \in L^\Psi(G)$ so that $L^\Phi(G,\omega) \subseteq L^1(G)$. This paper studies the spectral theory and primitive ideal structure of $L^\Phi(G,\omega)$.  In particular, we focus on studying the Hermitian, Wiener and $*$-regularity properties on this algebra, along with some related questions on spectral synthesis. It is shown that $L^\Phi(G,\omega)$ is always quasi-Hermitian, weakly-Wiener and $*$-regular. Thus, if $L^\Phi(G,\omega)$ is Hermitian, then it is also Wiener. Although, in general, $L^\Phi(G,\omega)$ is not always Hermitian, it is known that Hermitianness of $L^1(G)$ implies Hermitianness of $L^\Phi(G,\omega)$ if $\omega$ is sub-additive. We give numerous examples of locally elliptic groups $G$ for which $L^1(G)$ is Hermitian and sub-additive $L^\Phi$-weights on these groups. In the weighted $L^1$ case, even stronger Hermitianness results are formulated.


\end{abstract}

\maketitle


\section{Introduction}

It is a classical question in harmonic analysis and Banach algebra theory to determine for which locally compact groups $G$ is the Banach $*$-algebra $L^1(G)$ Hermitan/symmetric, Wiener and/or $*$-regular (\textit{c.f.}\ \cite{Lep84,Pal01}). Such a group $G$ is called Hermitian, Wiener or Boidol if, respectively, $L^1(G)$ is Hermitian, Wiener or $*$-regular. 

Classically, these problems were studied primarily in the context of connected locally compact groups and discrete groups. Every connected locally compact group can be approximated by connected Lie groups, and as a consequence, it can be shown that a connected locally compact group is Hermitian (resp.\ Wiener) if and only if its approximating Lie groups are all Hermitian (resp.\ Wiener) \cite[Section 4]{Lep76}. 

In the case of connected Lie groups, it is well known that every connected nilpotent Lie group is Hermitian, Wiener and Boidol, and for connected solvable Lie groups of real dimension $\le 4$, there is one exceptional group, the ``Poguntke group'' \cite[$\S$12.6.27]{Pal01}, which does not satisfy these properties. On the other hand, non-compact connected semisimple Lie groups are never Hermitian, Wiener or Boidol \cite[Chapter 12]{Pal01}.  

For discrete groups, it is a standard fact that these groups are ``weakly-Wiener'', and as a consequence, Hermitianness implies the Wiener property for discrete groups. It is shown in \cite[Section 4]{Lep76} that a solvable finitely-generated discrete group is Hermitian if and only if it has polynomial growth.

More general results of this form are also known: for example, nilpotent groups are always Hermitian and Wiener \cite{Lud79}, Boidol groups are always amenable \cite{BLSV78}, and it was recently proved that quasi-Hermitian groups, and hence Hermitian groups, are amenable too \cite{SW20}.

The Hermitian, Wiener and $*$-regularity properties can be formulated more generally for any Banach $*$-algebra (see \cite{Pal01}). A more general class of Banach $*$-algebras where these properties have been frequently studied are for weighted $L^1$-algebras, and more recently, weighted $L^p$-algebras, on locally compact groups \cite{Pyt73,Pyt82,FGLLMB03,DLMB04,Kuz06,Kuz08,Kuz09,KMB12}. In particular, if $G$ is a locally compact group, by a weight on $G$, we mean a measurable function $\omega: G \rightarrow [1,\infty)$ that is bounded on compact sets, sub-multiplicative and symmetric (see Definition \ref{dfn:weight} for more details). When $G$ is unimodular, these conditions on $\omega$ guarantee that the weighted space $L^1(G,\omega)$ is a Banach $*$-algebra with the convolution product and the involution coming from $L^1(G)$. For $p \in (1,\infty)$ and $q := \frac{p}{p-1}$, if $\omega$ further satisfies the property that $\omega^{-q}*\omega^{-q} \le \omega^{-q}$, then the space $L^p(G,\omega)$ is a Banach $*$-algebra under convolution and the same involution as in the $L^1$ case. Such a weight is known to exist for any $\sigma$-compact group and $p \in (1,\infty)$ \cite[Theorem 1.1]{Kuz08}. 

Weighted algebras of these forms have been studied for many years, and in the case of abelian groups, they are well understood and have many connections with classical Fourier analysis and Banach algebra theory. See for example \cite{Dom56} and the references there in.

A celebrated result in the direction of weighted $L^p$-algebras is that if $G$ is a compactly generated group of polynomial growth, and $\omega$ a weight on $G$ satisfying some technical growth conditions, then $L^1(G,\omega)$ is Hermitian, Wiener and $*$-regular \cite{FGLLMB03,DLMB04}. Similar results have also been obtained for weighted $L^p$-algebras ($p \in (1,\infty)$) on compactly generated groups of polynomial growth \cite{KMB12}, however, some difficulties in determining which of these algebras are Hermitian prevent the results from being as general as in the $L^1$ case. This line of work relies critically on the proof by Losert that every compactly generated group of polynomial growth is Hermitian, which is a consequence of his structure theory for such groups \cite{Los01}. The work is also motivated by and refers to some older results of Hulanicki, Pytlik and Dixmier on the spectral theory and functional calculus of these algebras \cite{Dix60,Hul66,Hul72,Pyt73,Pyt82}.

Even more recently, these results about weighted $L^p$-algebras on compactly generated groups of polynomial growth were generalised and extended in the setting of (twisted) weighted Orlicz $*$-algebras \cite{OO15,OS17}. An Orlicz space is a generalisation of an $L^p$-space which was defined in the 1930's by Orlicz \cite{Orl32}. Both $L^p$-spaces and variable Lebesgue spaces provide standard examples of Orlicz spaces, and certain Sobolev spaces can be found as subspaces of Orlicz spaces. Analysts have been interested in various classes of Orlicz spaces over the years, due to, for example, their applications in partial differential equations, calculus of variations and physics \cite{CF13,ML14,HH19}. Thus, Orlicz spaces form a natural generalisation of $L^p$-spaces to study.

To be more explicit, one constructs Orlicz spaces as follows: given a measure space $(X,\mu)$ and a Young function $\Phi: \mathbb{R} \rightarrow [0,\infty]$, one can associate a certain Banach space $L^\Phi(X)$ of measurable complex valued functions on $X$, and it is this space that we call an \textit{Orlicz space} (see Section \ref{sec:Orc} for the complete definition or \cite{RR91} for further theoretical details). The case when $\Phi(x) = \lvert x \rvert^p/p$ gives the classical $L^p$-spaces. Now, given a locally compact group $G$ and a Young function $\Phi$, one can study the space $L^\Phi(G)$. Typically $L^\Phi(G)$ is not closed under convolution, even when this is an $L^p$-space, but if $G$ is unimodular and $\omega$ a suitable weight on $G$, the weighted Orlicz space $L^\Phi(G,\omega)$ may be a Banach $*$-algebra when equipped with the convolution product and the involution $f^*(x):=\overline{f(x^{-1})}$ ($f \in L^\Phi(G,\omega)$). This is what we will refer to as a \textit{weighted Orlicz $*$-algebra} (or sometimes a \textit{weighted $L^\Phi$-algebra}) and we will call such a weight $\omega$ an \textit{$L^\Phi$-weight}.

In this paper we study the harmonic analysis of weighted Orlicz $*$-algebras on locally elliptic groups. A locally compact group is locally elliptic if and only if it can be written as a countable ascending union of compact open subgroups. In particular, any non-compact locally elliptic group cannot be compactly generated, however, these groups have polynomial growth. This work, in particular, tests the necessity of the compactly generated assumption in the work on compactly generated groups of polynomial growth described in the previous paragraphs. Furthermore, every (not necessarily compactly generated) locally compact group of polynomial growth is the extension of a locally elliptic group by a Lie group \cite[Theorem 3.3]{Los21}, so understanding weighted Orlicz $*$-algebras on general groups of polynomial growth would naturally require one to understand what happens in the locally elliptic case. 

Locally elliptic groups also feature in many places throughout the theory of totally disconnected locally compact (tdlc) groups and they thus form an interesting class of groups to study the harmonic analysis of from the perspective of tdlc group theory. For example, locally elliptic groups have strong connections with the theory of contraction groups \cite{BW04,GW10,GW21,GW21b} and the theory of scale groups \cite{BW04,Hor15,Wil20}. These two classes of groups are actively studied in tdlc group theory and play an important role. Also, any unipotent linear algebraic group over a non-archimedean local field is locally elliptic, so there is potential, as a consequence of the results in this paper, to construct an analogous theory of spectral synthesis and weighted algebras on such groups as there is in the case of connected nilpotent Lie groups \cite{Lud83b,Lud83a,LMB10,LMBP13,BL16}.

The motivation and goals of the present paper are the following:

\begin{enumerate}[(i)]
   \item To show that many of the arguments and results for weighted Orlicz $*$-algebras on compactly generated groups of polynomial growth hold equally well for locally elliptic groups. In certain aspects, the theory in the locally elliptic case is even cleaner than the compactly generated case.
   \item Give new examples of weighted Orlicz $*$-algebras that have nice Banach algebra properties, such as being (quasi-)Hermitian, (weakly-)Wiener and $*$-regular, and provide a range of examples of groups and weights that fit into this theory.
   \item Initiate further research into understanding the harmonic analysis of tdlc groups and present the work in a way that is relatively accessible to researchers in tdlc group theory.
\end{enumerate}

As a consequence of the above motives, particular (iii), we choose to be more elaborate in our exposition so that the paper is accessible to researchers who are not experts in the theory of Banach algebras nor familiar with this line of research. In particular, we note well that a number of the arguments in this paper are already well known and written down in the literature, but we choose to include them here for completeness and expository purposes. 

After going through some preliminaries and introductory material on Banach $*$-algebras, Orlicz spaces and locally elliptic groups in Section \ref{sec:2}, we study the properties of weights on locally elliptic groups in Section \ref{sec:3}, which is critical to our later results. In particular, in Proposition \ref{prop:GRSw}, we show that every weight on a locally elliptic group satisfies the GRS condition, which is an important condition used for obtaining Hermitianness of weighted $L^1$-algebras in the context of compactly generated groups of polynomial growth \cite{FGL06}. We also show that every weight $\omega$ on a locally elliptic group is dominated by a sub-additive weight $\omega^\sharp_1$ such that $1/\omega^\sharp_1 \in L^1(G) \cap L^\infty(G)$, see Proposition \ref{prop:polyw1}. Since $\omega^\sharp_1$ is sub-additive and $1/\omega^\sharp_1 \in L^1(G) \cap L^\infty(G)$, it can be shown that $L^\Phi(G,\omega^\sharp_1)$ is a Banach $*$-algebra for any Young function $\Phi$ \cite[Theorem 4.5]{OS17}. In particular, every weight on a locally elliptic group is dominated by a sub-additive $L^\Phi$-weight, for any given Young function $\Phi$. This later property is used extensively throughout the article.

To construct the weight $\omega^\sharp_1$ as in Proposition \ref{prop:polyw1}, we need to write the locally elliptic group $G$ as an ascending union of compact open subgroups $G = \bigcup_{i=1}^\infty K_i$ such that the sequence of indices $([K_{i+1}:K_i])_{i=1}^\infty$ is non-decreasing. Such a set of compact open subgroups $(K_i)_{i\in \mathbb{N}}$ satisfying the property that the sequence of indices $([K_{i+1}:K_i])_{i=1}^\infty$ is non-decreasing will be called a \hlight{standard decomposition} of $G$. It is obvious that any locally elliptic group has a standard decomposition.

The following result is a consequence of the above results on weights. We note that the assumption in the following theorem that $1/\omega \in L^\Psi(G)$ implies $L^\Phi(G,\omega) \subseteq L^1(G)$ by the Hölder inequality for Orlicz spaces \cite[Section 3.3]{RR91}. Also, we use $\nu_\mathcal{A}$ to denote the spectral radius function on a Banach algebra $\mathcal{A}$.

\begin{thm}\label{thm:1}
Let $G$ be a locally elliptic group, $(K_i)_{i \in \mathbb{N}}$ a standard decomposition of $G$, $(\Phi,\Psi)$ a complementary pair of Young functions and $\omega$ an $L^\Phi$-weight on $G$ with $1/\omega \in L^\Psi(G)$. Define $\omega_1^\sharp$ to be the weight constructed in Proposition \ref{prop:polyw1} with respect to the compact open subgroups $(K_i)_{i\in \mathbb{N}}$. Then, for all $f \in L^\Phi(G,\omega^\sharp_1) \subseteq L^\Phi(G,\omega)$, $\nu_{L^\Phi(G,\omega)}(f) = \nu_{L^1(G)}(f)$.
\end{thm} 

As already mentioned earlier, as a consequence of Losert's work \cite{Los01}, if $G$ is a compactly generated group of polynomial growth, $L^1(G)$ is always Hermitian. Then, if $\omega$ is a weight on $G$, $L^1(G,\omega)$ is Hermitian if and only if $\omega$ satisfies the GRS condition \cite[Theorem 1.3]{FGL06}. In contrast, if we now let $G$ be a locally elliptic group, $L^1(G)$ is not always Hermitian (see Section \ref{sec:4}), however, it is quasi-Hermitian as a consequence of \cite[Remark 4.10]{SW20}. Part (i) of the following theorem is then a consequence of Theorem \ref{thm:1}, but it can also been proved via the fact that every weight on a locally elliptic group satisfies the GRS condition. Part (ii) follows from \cite[Theorem 4.5]{OS17}. The proof of Theorem \ref{thm:2} is found in Section \ref{sec:4}. In the following, we use $\sigma_\mathcal{A}(x)$ to denote the spectrum of an element $x$ in a Banach algebra $\mathcal{A}$.

\begin{thm}\label{thm:2}
Let $G$ be a locally elliptic group, $(\Phi,\Psi)$ a complementary pair of Young functions and $\omega$ an $L^\Phi$-weight on $G$ with $1/\omega \in L^\Psi(G)$. The following hold:
\begin{enumerate}[(i)]
   \item  The Banach $*$-algebra $L^\Phi(G,\omega)$ is quasi-Hermitian i.e.\ for every self-adjoint function $f \in C_c(G) \subseteq L^\Phi(G,\omega)$, $\sigma_{L^\Phi(G,\omega)}(f) \subseteq \mathbb{R}$; \
   \item If $G$ is Hermitian and $\omega$ sub-additive, then for all $f \in L^\Phi(G,\omega)$, $\sigma_{L^\Phi(G,\omega)}(f) = \sigma_{L^1(G)}(f)$. In particular, $L^\Phi(G,\omega)$ is Hermitian.
\end{enumerate}
\end{thm}

Although a locally elliptic group $G$ is not always Hermitian, one can always construct a weight $\omega$ on $G$ such that $L^1(G,\omega)$ is Hermitian (see Section \ref{sec:3} and Section \ref{sec:4}). Also, we give a number of examples of Hermitian locally elliptic groups and sub-additive $L^\Phi$-weights on these groups in Section \ref{sec:8}.

In Section \ref{sec:5}, we prove that every weighted Orlicz $*$-algebra on a locally elliptic group is weakly-Wiener. The proof requires one to construct a functional calculus for certain smooth periodic functions. This functional calculus is also constructed in Section \ref{sec:5} and makes use of our results about weights on locally elliptic groups.

\begin{thm}\label{thm:3}
Let $G$ be a locally elliptic group, $(\Phi,\Psi)$ a complementary pair of Young funcitons and $\omega$ an $L^\Phi$-weight on $G$ such that $1/\omega \in L^\Psi(G)$. The Banach $*$-algebra $L^\Phi(G,\omega)$ is weakly-Wiener. In particular, if $L^\Phi(G,\omega)$ is Hermitian, then it is Wiener.
\end{thm}

Specialising the above theorem to the case of $L^1(G)$, we also see that every locally elliptic group is weakly-Wiener. We note this was already known since \cite{Lud79}, but we list it here to be explicit.

\begin{cor}
Every locally elliptic group is weakly-Wiener. In particular, every Hermitian locally elliptic group is Wiener.
\end{cor}

Section \ref{sec:6} studies the representation theory, $*$-regularity property and $C^*$-enveloping algebras of a weighted Orlicz $*$-algebra on a locally elliptic group. The main result is the following theorem.
In the statement of the theorem, the condition that $\Phi \in \Delta_2$ is defined in Definition \ref{dfn:delta2}, and this assumption, for example, guarantees that $L^\Phi(G,\omega)^* = L^\Psi(G,\omega^{-1})$, where $\Psi$ is the complementary Young function to $\Phi$.

\begin{thm}\label{thm:4}
Let $G$ be a locally elliptic group, $(\Phi,\Psi)$ a complementary pair of Young functions with $\Phi \in \Delta_2$, and $\omega$ an $L^\Phi$-weight on $G$ such that $1/\omega \in L^\Psi(G)$. The following hold:
\begin{enumerate}[(i)]
   \item $C^*(L^\Phi(G,\omega)) \cong C^*(G)$;\
   \item The algebra $L^\Phi(G,\omega)$ is $*$-regular, in particular, $\Prim_*(L^\Phi(G,\omega))$, $\Prim_*(L^1(G))$ and $\Prim(C^*(G))$ are homeomorphic.
\end{enumerate}
\end{thm}

This result could be useful in the study of the unitary representation theory of locally elliptic groups. Indeed, for example, it is currently an open question for which locally elliptic groups $G$ does the topology on $\Prim(C^*(G))$ satisfy the $T_1$ separation axiom. Sometimes $C^*(G)$ is not the most convenient algebra to work with: being a completion of $L^1(G)$, one cannot treat all elements of $C^*(G)$ as functions on $G$. Thus, it is often convenient to work with $L^1(G)$ instead. In the case of a locally elliptic group, as given by the above theorem, $\Prim_*(L^1(G))$ and $\Prim(C^*(G))$ are homeomorphic, so it is fine working with $L^1(G)$ instead of $C^*(G)$ in the study of the unitary representation theory. However, unlike $C^*(G)$, $L^1(G)$ may not be Hermitian, which is a desirable property to have. But even if $L^1(G)$ is not Hermitian, as is shown in the article, there always exist weights on $G$ such that $L^1(G,\omega)$ is Hermitian (and $\Prim_*(L^1(G,\omega)) \cong \Prim(C^*(G)))$. So in the context of non-Hermitian locally elliptic groups, working with a weighted $L^1$-algebra may be more natural algebra than working with $C^*(G)$ or $L^1(G)$ when studying the unitary representation theory of these groups.

In Section \ref{sec:7}, we show that every hull of a weighted Orlicz $*$-algebra contains a minimal ideal. This leads to interesting questions to pursue in the context of spectral synthesis, which we discuss in Section \ref{sec:8}.
In Section \ref{sec:8}, we also go through a variety of examples of locally elliptic groups and weighted algebras on these groups. Furthermore, we pose some open questions that could be investigated in future work.

\section{Preliminaries}\label{sec:2}

\subsection{Banach $*$-algebras and representation theory}

We begin the preliminaries section by collecting some notation, definitions and results on Banach algebras that will be used throughout the article. We assume the reader has some familiarity with (Banach) $*$-algebras, as can be found in \cite{Dix77,Mur90}, for example.

\begin{notation}
Let $\mathcal{A}$ be a Banach $*$-algebra and $x \in \mathcal{A}$. 
\begin{enumerate}[(i)]
   \item $\norm{x}_{\mathcal{A}}$ denotes the \hlight{norm} of $x$ in $\mathcal{A}$.\
   \item $\sigma_\mathcal{A}(x)$ denotes the \hlight{spectrum} of $x$ in $\mathcal{A}$.\
   \item $\nu_\mathcal{A}(x)$ denotes the \hlight{spectral radius} of $x$ in $\mathcal{A}$.
\end{enumerate}
\end{notation}

We now define some notation and terminology in regards to the representation theory of Banach $*$-algebras.

\begin{notation}
Let $\mathcal{A}$ be a Banach $*$-algebra.
\begin{enumerate}[(i)]
   \item A \hlight{unitary representation} of $\mathcal{A}$ is a pair $(\pi, \mathcal{H})$, where $\mathcal{H}$ is a Hilbert space, and $\pi: \mathcal{A} \rightarrow \mathcal{B}(\mathcal{H})$ is a non-degenerate $*$-homomorphism to the $C^*$-algebra $\BH$ of bounded operators on $\H$. \
   \item $\widehat{\mathcal{A}}$ denotes the set of all topologically irreducible unitary representations of $\mathcal{A}$ upto unitary equivalence equipped with the Fell topology.\
   \item $\Prim_*(\mathcal{A})$ denotes the space of kernels of topological irreducible unitary representations of $\mathcal{A}$ equipped with the hull-kernel topology. \
   \item $\Prim(\mathcal{A})$ denotes the space of annihilators of simple $\mathcal{A}$-modules equipped with the hull-kernel topology. \
\end{enumerate}
\end{notation}

We will now remind the reader of the definition of the hull-kernel topology and Fell topology: given an arbitrary subset $X \subseteq \mathcal{A}$, define $\hull_*(X) := \{ I \in \Prim_*(\mathcal{A}) : X \subseteq I \}$ (resp.\ $\hull(X) := \{ I \in \Prim(\mathcal{A}) : X \subseteq I \})$. Then, the \hlight{hull-kernel topology} on $\Prim_*(\mathcal{A})$ (resp.\ $\Prim(\mathcal{A})$) is the topology generated by defining the sets $X \subseteq \Prim_*(\mathcal{A})$ (resp.\ $X \subseteq \Prim(\mathcal{A})$)  satisfying $\hull_*(\ker(X)) = X$ (resp.\ $\hull(\ker(X)) = X$) to be closed, where $\ker(X) = \bigcap_{I \in X} I$. The Fell topology on $\widehat{\mathcal{A}}$ is precisely the pullback of the hull-kernel topology on $\Prim_*(\mathcal{A})$ via the canonical surjection $\widehat{\mathcal{A}} \twoheadrightarrow \Prim_*(\mathcal{A})$.

The notion of a $C^*$-enveloping algebra is critical to understanding the $*$-regularity property.

\begin{dfn}
Let $\mathcal{A}$ be a Banach $*$-algebra. 
\begin{enumerate}[(i)]
   \item The \hlight{maximal $C^*$-norm} on $\mathcal{A}$ is the norm defined on $x \in \mathcal{A}$ by 
   \begin{displaymath} \norm{x}_{\text{max}} := \sup \{ \norm{\pi(x)}_{\B(\H_\pi)} : \pi \in \widehat{\mathcal{A}} \}.  \end{displaymath}
   \item The \hlight{reducing ideal} of $\mathcal{A}$, denoted $\mathcal{A}_R$, is the $*$-ideal in $\mathcal{A}$ consisting of all elements $x \in \mathcal{A}$ with $\norm{x}_{\text{max}} = 0$.
   \item The \hlight{enveloping $C^*$-algebra} of $\mathcal{A}$, denoted $C^*(\mathcal{A})$, is the completion of $\mathcal{A}/\mathcal{A}_R$ with respect to the maximal $C^*$-norm.
\end{enumerate}
\end{dfn}

The following proposition is fundamental to understanding the Hermitian property for a Banach $*$-algebra. We recall that an element $x \in \mathcal{A}$ of a Banach $*$-algebra $\mathcal{A}$ is \hlight{self-adjoint} if $x = x^*$.

\begin{prop}\label{prop:sym}\cite{Lep76b}
Let $\mathcal{A}$ be a Banach $*$-algebra. The following are equivalent:
\begin{enumerate}[(i)]
   \item For all $x \in \mathcal{A}$, $\sigma_\mathcal{A}(x^*x) \subseteq \mathbb{R}_{\ge 0}$; \
   \item For all self-adjoint $x \in \mathcal{A}$, $\sigma_\mathcal{A}(x) \subseteq \mathbb{R}$; \
   \item $\Prim(\mathcal{A}) \subseteq \Prim_*(\mathcal{A})$ i.e.\ every simple $\mathcal{A}$-module is unitarizable.
\end{enumerate}
\end{prop}

We now define the Hermitian, Wiener and $*$-regularity properties that this article focuses on studying.

\begin{dfn}
Let $\mathcal{A}$ be a Banach $*$-algebra. 
\begin{enumerate}[(i)]
   \item $\mathcal{A}$ is called \hlight{Hermitian} if one of the equivalent conditions of Proposition \ref{prop:sym} hold. \
   \item $\mathcal{A}$ is called \hlight{Wiener} if for every proper closed two-sided ideal $I \subseteq \mathcal{A}$, there exists $J \in \Prim_*(\mathcal{A})$ such that $I \subseteq J$.
   \item $\mathcal{A}$ is called \hlight{weakly-Wiener} if for every proper closed two-sided ideal $I \subseteq \mathcal{A}$, there exists $J \in \Prim(\mathcal{A})$ such that $I \subseteq J$.
   \item $\mathcal{A}$ is called \hlight{$*$-regular} if $\Prim_*(\mathcal{A})$ and $\Prim(C^*(\mathcal{A})) = \Prim_*(C^*(\mathcal{A}))$ are homeomorphic.
\end{enumerate}
\end{dfn}

\subsection{Orlicz spaces}\label{sec:Orc}

In this subsection we will introduce basic concepts concerning Orlicz spaces and Orlicz $*$-algebras, and give some results that will be used throughout the article. We primarily follow the references \cite{RR91,OO15,OS17} and the reader can consult these for further details on the topic.

\begin{dfn}
A \hlight{Young function} is a function $\Phi:\mathbb{R} \rightarrow [0,\infty]$ which satisfies the following conditions:
\begin{enumerate}[(i)]
   \item $\Phi$ is convex;\
   \item $\Phi$ is even;\
   \item $\Phi(0)=0$;\
   \item $\lim_{x \rightarrow \infty} \Phi(x) = + \infty$.
\end{enumerate}

One associates to the Young function $\Phi$ its \hlight{complementary Young function} defined by
\begin{displaymath}  \Psi(y) := \sup\{x\lvert y \rvert - \Phi(x) : x \ge 0 \}. \end{displaymath}
It can be checked that $\Psi$ is also a Young function and the complementary Young function to $\Psi$ is $\Phi$. We call $(\Phi,\Psi)$ a \hlight{complementary pair} of Young functions. 
\end{dfn}

A complementary pair of Young functions $(\Phi,\Psi)$ satisfies Young's inequality:
\begin{displaymath} xy \le \Phi(x) + \Psi(y) \; \; \; \; \forall x,y \in \mathbb{R}. \end{displaymath}

The following notion of the $\Delta_2$-condition for a Young function will be important later. We will mention later some consequences of this condition.

\begin{dfn}\label{dfn:delta2}
Let $\Phi$ be a Young function. We say that $\Phi$ satisfies the \hlight{$\Delta_2$-condition}, or $\Phi \in \Delta_2$, if there exists a constant $C > 0$ such that 
\begin{displaymath} \Phi(2x) \le C \Phi(x) \end{displaymath}
for all $x \ge 0$.
\end{dfn}

We now define the Orlicz space associated to a Young function. For the remainder of this subsection, unless otherwise stated, $G$ will be a locally compact group and all integration is performed against some prior fixed left-Haar measure on $G$.

\begin{dfn}
Let $(\Phi,\Psi)$ be a complementary pair of Young functions. The \hlight{Orlicz space} on $G$ associated to $\Phi$ is the space
\begin{displaymath} L^\Phi(G) := \bigg\{ f:G \rightarrow \mathbb{C}\text{ measurable} : \int_G \Phi(\alpha\lvert f \rvert) \: dx < \infty \text{ for some }\alpha > 0 \bigg\}. \end{displaymath}
The space $L^\Phi(G)$ is equipped with the \hlight{Orlicz norm}
\begin{displaymath} \norm{f}_{L^\Phi(G)} := \sup\bigg\{ \int_G \lvert f(x)g(x) \rvert \: dx : \int_G \Psi(\vert g(x) \rvert) \: dx \le 1 \bigg\} \; \; \; \; (f \in L^\Phi(G))   \end{displaymath}
with respect to which $L^\Phi(G)$ becomes a Banach space. The Orlicz norm is equivalent to the \hlight{Luxemburg norm} which is defined as 
\begin{displaymath} N_\Phi(f) := \inf\bigg\{ k > 0 : \int_G \Phi(\lvert f(x) \rvert / k) \: dx \le 1 \bigg\} \; \; \; \; (f \in L^\Phi(G)).  \end{displaymath} 
\end{dfn}

We note that in the case that $\Phi(x) = \lvert x \rvert^p/p$ ($p \in (1,\infty)$), the complementary Young function to $\Phi$ is $\Psi =  \lvert x \rvert^q/q$ where $1/p + 1/q=1$, and these Young functions satisfy the $\Delta_2$ condition. Furthermore, in this case, $L^\Phi(G)$ becomes the classical $L^p$-space denoted $L^p(G)$, and the Orlicz norm with respect to the Young function $\Phi$ is equivalent to the usual $L^p$-norm on $L^p(G)$.

The following provides some other examples of Young functions, some of which arise in physics and probablilty theory.

\begin{exm}\cite{ML14}
The following are Young functions:
\begin{enumerate}[(i)]
   \item $\Phi = \lvert x \rvert^p/p$ ($p \in [1,\infty)$);\
   \item $\Phi(x) = e^{\lvert x \rvert} - \lvert x \rvert -1$; \
   \item $\Phi(x) = \cosh(x) -1$; \
   \item $\Phi(x) = x\log(1+x)$.
\end{enumerate}
\end{exm}

We now move on to studying weights and weighted Orlicz spaces.

\begin{dfn}\label{dfn:weight}
A \hlight{weight} $\omega$ on a locally compact group $G$ is a measurable function $\omega: G \rightarrow [1,\infty)$ that satisfies the following properties:
\begin{enumerate}[(i)]
   \item $\omega$ is bounded on compact sets; \
   \item $\omega$ is \hlight{sub-multiplicative} i.e.\ $\omega(xy) \le \omega(x)\omega(y)$ for all $x,y \in G$;\
   \item $\omega$ is \hlight{symmetric} i.e.\ $\omega(x) = \omega(x^{-1})$ for all $x \in G$.
\end{enumerate}
Two weights $\omega$ and $\omega'$ on $G$ are \hlight{equivalent} if there exists constants $C$ and $C'$ such that $C \omega(x) \le \omega'(x) \le C'\omega(x)$ for all $x \in G$.
\end{dfn}

The notion of a sub-additive weight will be used throughout the article. A sub-additive weight is also often referred to as a polynomial weight, which was defined by Pytlik in \cite{Pyt82}.

\begin{dfn}
Let $\omega$ a weight on $G$. The weight $\omega$ is called \hlight{sub-additive} if there exists a constant $C >0$ such that for all $x,y \in G$
\begin{displaymath} \omega(xy) \le C(\omega(x) + \omega(y)). \end{displaymath}
\end{dfn}

We now define the notion of a weighted Orlicz space.

\begin{dfn}
Let $\Phi$ be a Young function and $\omega$ a weight on $G$. The \hlight{weighted Orlicz space} on $G$ corresponding to $\Phi$ and $\omega$ is the space
\begin{displaymath} L^\Phi(G,\omega) := \{ f \in L^\Phi(G) : \norm{f\omega}_{L^\Phi(G)} < \infty \} \end{displaymath}
which becomes a Banach space when equipped with the norm 
\begin{displaymath} \norm{f}_{L^\Phi(G,\omega)} := \norm{f\omega}_{L^\Phi(G)} \; \; \; \; (f \in L^\Phi(G,\omega)). \end{displaymath}
\end{dfn}

We note that if we have two equivalent weights $\omega$ and $\omega'$ on $G$, then the corresponding weighted Orlicz spaces $L^\Phi(G,\omega)$ and $L^\Phi(G,\omega')$ are isomorphic as Banach spaces. Also, if $\Phi \in \Delta_2$, then one can show that the dual of $L^\Phi(G,\omega)$ is $L^\Phi(G,\omega)^* = L^\Psi(G,\omega^{-1})$, where $\Psi$ is the complementary Young function to $\Phi$. In particular, if $\Psi$ is also $\Delta_2$, then $L^\Phi(G,\omega)$ is a reflexive Banach space.

The following elementary fact will be used throughout the article.

\begin{prop}\label{prop:Orc1}
Let $G$ be a locally compact group, $(\Phi,\Psi)$ a complementary pair of Young functions, and $\omega$ a weight on $G$. If $1/\omega \in L^\Psi(G)$ then $L^\Phi(G,\omega) \subseteq L^1(G)$.
\end{prop}

\begin{proof}
The proof follows directly from the Hölder inequality for Orlicz spaces \cite[Section 3.3]{RR91}.
\end{proof}

The following result is critical to our article.

\begin{prop}\label{prop:Orc2}\cite[Theorem 4.5]{OS17}
Let $G$ be a unimodular locally compact group, $(\Phi,\Psi)$ a complementary pair of Young functions and $\omega$ a weight on $G$. If $\omega$ is sub-additive and $1/\omega \in L^\Psi(G)$, then $L^\Phi(G,\omega) \subseteq L^1(G)$ is a Banach $*$-algebra when equipped with the convolution product and the involution $f^*(x) := \overline{f(x^{-1})}$ ($f\in L^\Phi(G,\omega)$). Moreover, the following hold:
\begin{enumerate}[(i)]
   \item There exists a constant $C > 0$ such that for all $f, g \in L^\Phi(G,\omega)$
   \begin{displaymath} \norm{f*g}_{L^\Phi(G,\omega)} \le C(\norm{f}_{L^1(G)}\norm{g}_{L^\Phi(G,\omega)} + \norm{f}_{L^\Phi(G,\omega)}\norm{g}_{L^1(G)}); \end{displaymath}
   \item If $L^1(G)$ is Hermitian, then $L^\Phi(G,\omega)$ is Hermitian.
\end{enumerate}
\end{prop}

It can also be the case that $L^\Phi(G,\omega)$ becomes a Banach $*$-algebra under convolution even when $\omega$ is not sub-additive. We make the following definition for ease of terminology throughout this paper.

\begin{dfn}
Let $\Phi$ be a Young function and $\omega$ a weight on a unimodular locally compact group $G$. We call $\omega$ an \hlight{$L^\Phi$-weight} if $L^\Phi(G,\omega)$ is a Banach $*$-algebra when equipped with the convolution product and the involution $f^*(x) = \overline{f(x^{-1})}$.
\end{dfn}


To finish the section on Orlicz spaces, we state some standard facts about Orlicz algebras that will be used in the article. The proofs can be found in the papers \cite{OO15,OS17}.

\begin{prop}\label{prop:Orc3}\cite{OO15,OS17}
Let $G$ be a unimodular locally compact group with left Haar measure $\mu$, $(\Phi,\Psi)$ a complementary pair of Young functions and $\omega$ a $L^\Phi$-weight on $G$. The following hold:
\begin{enumerate}[(i)]
  \item For $x \in G$, let $L_x$ be the operator on $L^\Phi(G,\omega)$ defined by $L_x(f)(y) := f(x^{-1}y)$ for $f \in L^\Phi(G,\omega)$. Then, $\norm{L_xf}_{L^\Phi(G,\omega)} \le \omega(x) \norm{f}_{L^\Phi(G,\omega)}$ for all $f \in L^\Phi(G,\omega)$.
  \item Let $\mathcal{N}$ be the set of compact neighbourhoods of the identity in $G$. Then, $(\chi_K/\mu(K))_{K \in \mathcal{N}}$, where $\chi_K$ denotes the characteristic function on the set $K$, is an approximate identity in $L^\Phi(G,\omega)$. It is a bounded approximate identity in $L^1(G,\omega)$.
  \item The space $L^\Phi(G,\omega)$ is a $L^1(G,\omega)$-module with respect to convolution. For $f \in L^1(G,\omega)$ and $g \in L^\Phi(G,\omega)$, we have that
  \begin{displaymath} \norm{f*g}_{L^\Phi(G,\omega)} \le \norm{f}_{L^1(G,\omega)} \norm{g}_{L^\Phi(G,\omega)}.  \end{displaymath}
\end{enumerate}
\end{prop}

%

\subsection{Locally elliptic locally compact groups}

First we define the notion of a locally elliptic group.

\begin{dfn}
Let $G$ be a locally compact group. The group $G$ is called \hlight{locally elliptic} if every compact subset of $G$ generates a relatively compact subgroup.
\end{dfn}

Of course, if $G$ is discrete, then locally elliptic is synonymous with the term locally finite. One also checks easily that every locally elliptic group is unimodular.

The following result, which is a consequence of work of Platonov, gives an equivalent characterisation of locally elliptic groups.

\begin{thm}\cite{Pla66}
Let $G$ be a locally compact group. Then $G$ is locally elliptic if and only if it is a countable increasing union of compact open subgroups.
\end{thm}

\textbf{Throughout the remainder of the article, if we say that $G = \bigcup_{i=1}^\infty K_i$ is a locally elliptic group, then we implicitly assume that $G$ is locally compact, the $K_i$ are compact open subgroups of $G$, and $K_i \le K_{i+1}$ for all $i \in \mathbb{N}$. Furthermore, we assume that $G$ is not compact to avoid any trivialities.} \\

We give numerous examples of locally elliptic groups and their connection with the theory in this paper in Section \ref{sec:8}.


\section{Weights on locally elliptic groups}\label{sec:3}

In this section we will study some properties concerning weights on locally elliptic groups that will be critical to our results later in the article. 

In the context of weighted $L^1$-algebras on compactly generated groups of polynomial growth, the condition of a weight satisfying the GRS condition is an important part of the theory. Indeed, a weighted $L^1$-algebra $L^1(G,\omega)$ on a compactly generated group of polynomial growth $G$ is Hermitian if and only if $\omega$ satisfies the GRS condition \cite[Theorem 1.3]{FGL06}.

\begin{dfn}
Let $G$ be a locally compact group and $\omega$ a weight on $G$. The weight $\omega$ is said to satisfy the \hlight{GRS condition}, if for all $x \in G$, $\lim_{n \rightarrow \infty} \omega(x^n)^{1/n} = 1$.
\end{dfn}

We will now show that every weight on a locally elliptic group satisfies the GRS condition.

\begin{prop}\label{prop:GRSw}
Let $G$ be a locally elliptic group and $\omega$ a weight on $G$. Then $\omega$ satisfies the GRS condition.
\end{prop}

\begin{proof}
Let $x \in G$. Since $G$ is locally elliptic, the set $\{ x^n : n \in \mathbb{N} \}$ is contained in some compact open subgroup $K \le G$. By definition, $\omega$ is bounded on compact sets, so $C:= \sup\{ \omega(k) : k \in K \}$ is finite. Then, for all $n \in \mathbb{N}$, $\omega(x^n) \le C$. Thus we have that 
\begin{displaymath} \lim_{n \rightarrow \infty} \omega(x^n)^{1/n} \le \lim_{n \rightarrow \infty} C^{1/n} = 1.\end{displaymath}
\end{proof}

We now define a class of weights that exist on a locally elliptic group. These weights will be important throughout the article.

\begin{dfn}
Let $G = \bigcup_{i=1}^\infty K_i$ be a locally elliptic group and $\bold{a} := (a_i)_{i=1}^\infty \subseteq \mathbb{R}_{\ge 1}$ a sequence. Define a function $\omega_{\bold{a}}$ on $G$ by
\begin{displaymath} \omega_{\bold{a}} := a_1\chi_{K_1} + \sum_{i=2}^\infty a_i \chi_{K_i \setminus K_{i-1}} \end{displaymath}
where $\chi_U$ denotes the characteristic function on the set $U \subseteq G$. If $\bf{a}$ is a non-decreasing sequence, one checks easily that $\omega_{\bold{a}}(xy) \le \max\{\omega_{\bold{a}}(x),\omega_{\bold{a}}(y)\}$ for all $x,y \in G$, from which it follows that $\omega_{\bf{a}}$ is sub-additive weight on $G$.
\end{dfn}

We now show that any weight on a locally elliptic group $G$ is dominated by a weight of the form $\omega_\bold{a}$. Furthermore, we show that $\bold{a}$ can be chosen so that $\omega_\bold{a}$ is an $L^\Phi$-weight for any Young function $\Phi$. We use the following lemma in the proof of the next proposition.

\begin{lem}\label{lem:lpw}\cite[Korollar 3.8]{Fei79}
Let $G$ be a unimodular locally compact group and $\omega$ a locally-integrable sub-additive weight on $G$. If $p \in (1,\infty)$ and $\omega^{-1} \in L^{q}(G)$ where $q := \frac{p}{p-1}$, then $L^p(G,\omega)$ is a Banach $*$-algebra.
\end{lem}

We note that there is a mistake in the statement of \cite[Korollar 3.8]{Fei79} but we have corrected it in our version. Also, in \cite{Fei79}, weights are not assumed to be symmetric, so the conclusion in \cite[Korollar 3.8]{Fei79} is that $L^p(G,\omega)$ is a Banach algebra, not a Banach $*$-algebra.

\begin{prop}\label{prop:polyw1}
Let $G = \bigcup_{i=1}^\infty K_i$ be a locally elliptic group, $\mu$ a left Haar measure on $G$ normalised so that $\mu(K_1) = 1$, $p \in [1,\infty)$ and $\omega$ a weight on $G$. Set $a_i := \sup \{ \omega(x) : x \in K_i \}$ for each $i \in \mathbb{N}$. Define two functions on $G$ by
\begin{displaymath} \omega^\sharp := a_1\chi_{K_1} + \sum_{i=2}^\infty a_i \chi_{K_i \setminus K_{i-1}} \end{displaymath}
and
\begin{displaymath} \omega^\sharp_p := (a_1+1)\chi_{K_1} + \sum_{i=2}^\infty (a_i + i^2) \mu(K_i\setminus K_{i-1})^{1/q}\chi_{K_i \setminus K_{i-1}} \end{displaymath}
where $q = \frac{p}{p-1}$ if $p>1$, and $q =1$ if $p=1$.
Then, $\omega^\sharp$ is a sub-additive weight on $G$, $\omega \le \omega^\sharp \le \omega^\sharp_p$, and $1/\omega^\sharp_p \in L^{q}(G)$. Furthermore, if the sequence of indices  $([K_{i+1}:K_{i}])_{i=1}^\infty$ is non-decreasing, then $\omega^\sharp_p$ is a sub-additive $L^p$-weight on $G$.
\end{prop}

\begin{proof}

Since each of the $K_i$ are compact open subgroups of $G$ and $\omega$ is bounded on compact sets by definition, it follows that $a_i := \sup \{ \omega(x) : x \in K_i \}$ is finite for each $i$. Also, since $K_i \le K_{i+1}$ for each $i$, the sequence $(a_i)_{i=1}^\infty \subseteq \mathbb{R}_{\ge1}$ is non-decreasing, and hence $\omega^\sharp$ is a sub-additive weight on $G$ by the prior discussion. Similarly, if the sequence $([K_{i+1}:K_{i}])_{i=1}^\infty$ is non-decreasing, then the sequence $((a_i+i^2)\mu(K_i\setminus K_{i-1})^{1/q})_{i=1}^\infty$ (where we take $K_0$ to be the empty set) is also non-decreasing, and hence $\omega^\sharp_p$ is a sub-additive weight on $G$.

Since we assume that $K_{i}$ is a proper subgroup of $K_{i+1}$ for each $i$, it follows that 
\begin{align*} \mu(K_{i+1} \setminus K_{i}) &= \mu(K_{i+1}) - \mu(K_{i}) \\
& = [K_{i+1}:K_{i}]\mu(K_{i})  - \mu(K_{i}) \\ 
&\ge 2\mu(K_{i})  - \mu(K_{i})  \ge 1.
\end{align*}
It is then clear, by definition of $\omega^\sharp$ and $\omega^\sharp_p$, that $\omega \le \omega^\sharp \le \omega^\sharp_p$.

We now just need to show that $\omega^\sharp_p$ is an $L^p$-weight. First, we show that $1/\omega^\sharp_p \in L^{q}(G)$. To do this, one checks that
\begin{displaymath} 1/\omega^\sharp_p = \frac{\chi_{K_1}}{a_1+1} + \sum_{i=2}^\infty  \frac{\chi_{K_i \setminus K_{i-1}}}{(a_i + i^2)\mu(K_i\setminus K_{i-1})^{1/q}}\end{displaymath}
from which it follows that for any $q \in [1,\infty)$
\begin{align*} 
\int_G (1/\omega^\sharp_p(x))^{q} \: d\mu(x) &= \int_{K_1} (1/\omega^\sharp_p(x))^{q} \: d\mu(x) + \sum_{i=2}^{\infty} \int_{K_{i} \setminus K_{i-1}} (1/\omega^\sharp_p(x))^{q} \: d\mu(x) \\ 
&= \frac{1}{(a_1+1)^{q}} + \sum_{i=2}^\infty \frac{1}{(a_i+i^2)^{q}} \\
& \le \frac{1}{(a_1+1)^{q}} + \sum_{i=2}^\infty \frac{1}{i^{2q}} < \infty.
\end{align*}
Thus $1/\omega^\sharp_p \in L^{q}(G)$. It then follows by Lemma \ref{lem:lpw} that $\omega^\sharp_p$ is an $L^p$-weight on $G$. This completes the proof.
\end{proof}

We recall from the introduction that if $G$ is locally elliptic group, a \hlight{standard decomposition} of $G$ is a set of compact open subgroups $(K_i)_{i \in \mathbb{N}}$ of $G$ such that $K_i \le K_{i+1}$ for each $i$, $G = \bigcup_{i=1}^\infty K_i$ and the sequence of indices $([K_{i+1}:K_i])_{i=1}^\infty$ is non-decreasing. The following is then a direct corollary of the theorem.

\begin{cor}
Let $G$ be a locally elliptic group, $(K_i)_{i\in \mathbb{N}}$ a standard decomposition of $G$, $(\Phi,\Psi)$ a complementary pair of Young functions and $\omega$ a weight on $G$. Then, the weight $\omega^\sharp_1$, as constructed in the proposition with respect to the compact open subgroups $(K_i)_{i\in \mathbb{N}}$, is an $L^\Phi$-weight.
\end{cor}

\begin{proof}
Indeed, by Proposition \ref{prop:polyw1}, $1/\omega^\sharp_1 \in L^1(G) \cap L^\infty(G) \subseteq L^\Psi(G)$. It then follows from Proposition \ref{prop:Orc2} that $L^\Phi_{}(G,\omega^\sharp_1)$ is a Banach $*$-algebra and hence $\omega^\sharp_1$ is an $L^\Phi$-weight.
\end{proof}


We will now describe a method intrinsic to locally elliptic groups which allows one to determine if a weight is (equivalent to) a sub-additive weight.

\begin{dfn}
Let $G$ be a locally elliptic group and $(K_i)_{i\in\mathbb{N}}$ a sequence of compact open subgroups of $G$ such that $K_i \le K_{i+1}$ for each $i$ and $G = \bigcup_{i=1}^\infty K_i$. Given a weight $\omega$ on $G$, define the \hlight{variation} of $\omega$ with respect to the $(K_i)_{i \in \mathbb{N}}$ to be the value
\begin{displaymath} \text{Var}(\omega,(K_i)_{i\in \mathbb{N}}) := \sup_{i \in \mathbb{N}} (\max_{x \in K_i} \omega(x) - \min_{y \in K_i} \omega(y)). \end{displaymath}
We say that $\omega$ has \hlight{bounded variation} if there exists such a sequence of compact open subgroups $(K_i)_{i\in \mathbb{N}}$ in $G$ such that $\text{Var}(\omega, (K_i)_{i\in \mathbb{N}}) < \infty$.
\end{dfn}

We then have the following result.

\begin{prop}\label{prop:bvweight}
Let $G$ be a locally elliptic group and $\omega$ a weight on $G$. If $\omega$ has bounded variation then it is equivalent to a sub-additive weight.
\end{prop}

\begin{proof}
Assume that $G$ is locally elliptic and $\omega$ a weight on $G$ with bounded variation. Let $(K_i)_{i \in \mathbb{N}}$ be a sequence of compact open subgroups of $G$ such that $K_i \le K_{i+1}$ for each $i$, $G = \bigcup_{i=1}^\infty K_i$, and $\text{Var}(\omega,(K_i)_{i\in\mathbb{N}}) =: C < \infty$. Let $\omega^\sharp$ be the sub-additive weight constructed in Proposition \ref{prop:polyw1} with respect to the compact open subgroups $(K_i)_{i \in \mathbb{N}}$. Then one checks easily that for all $x \in G$
\begin{displaymath} \omega(x) \le \omega^\sharp(x) \le (C+1)\omega(x) \end{displaymath}
which, by definition, means that $\omega$ is equivalent to the sub-additive weight $\omega^\sharp$. 
\end{proof}

In the following example, we show that a weight on a locally elliptic group need not have bounded variation nor be sub-additive.

\begin{exm}
For $n \in \mathbb{N}$, let $C_n$ denote the cyclic group of order $n$. Define the group $G := \bigoplus_{n=1}^\infty C_{n}$. Clearly $G$ is a locally finite group and hence locally elliptic. We define a weight $\omega$ on $G$ as follows: let $x_{n}$ denote the canonical generator of the copy of $C_{n}$ in $G$ for each $n \in \mathbb{N}$. Then, given a general element $x = x_{n_1}^{m_1} \cdots x_{n_k}^{m_k} \in G$, where $1\le m_i \le n_i-1$ is an integer for each $i$, we define $\omega$ on $x$ by
\begin{displaymath} \omega(x) := (n_1)^{\max\{m_1,n_1-m_1\}} \cdots (n_k)^{\max\{m_k,n_k-m_k\}}. \end{displaymath} 
One checks that $\omega$ is a weight on $G$. The weight $\omega$ is not sub-additive since, for example,
\begin{displaymath} \frac{\omega(x_{2n}^nx_{4n}^{2n})}{\omega(x_{2n}^n) + \omega(x_{4n}^{2n})} = \frac{(2n)^n(4n)^{2n}}{(2n)^{n} + (4n)^{2n}} \end{displaymath}
is unbounded as $n \rightarrow \infty$. This also implies that $\omega$ cannot have bounded variation by the previous proposition.
\end{exm}


\section{The Hermitian property}\label{sec:4}

In this section we study the Hermitian and quasi-Hermitian properties for weighted Orlicz $*$-algebras on locally elliptic groups. We first start by proving Theorem \ref{thm:1} from the introduction. \\

\noindent\textit{Proof of Theorem \ref{thm:1}.}
The proof is similar to the proof of Lemma 3.7 and Lemma 3.8 in \cite{KMB12}. Indeed, since $1/\omega \in L^\Psi(G)$ and hence $L^\Phi(G,\omega) \subseteq L^1(G)$, it follow by an identical argument to that of \cite[Lemma 2.2]{OO15} that there exists a constant $C_1>0$ such that 
\begin{displaymath} \norm{f}_{L^1(G)} \le C_1 \norm{f}_{L^\Phi(G,\omega)} \end{displaymath}
for all $f \in L^\Phi(G,\omega)$. Then, given $f \in L^\Phi(G,\omega)$,
\begin{displaymath} \nu_{L^1(G)}(f) = \lim_{n \rightarrow \infty} \norm{f^{*n}}_{L^1(G)}^{1/n} \le \lim_{n \rightarrow \infty} C_1^{1/n} \norm{f^{*n}}_{L^\Phi(G,\omega)}^{1/n} = \nu_{L^\Phi(G,\omega)}(f).\end{displaymath}
Thus, we just need to show that $\nu_{L^\Phi(G,\omega)}(f) \le \nu_{L^1(G)}(f)$ for all $f \in L^\Phi(G,\omega^\sharp_1)$. We now fix $f \in L^\Phi(G,\omega^\sharp_1)$. By Proposition \ref{prop:Orc2}(i), since $\omega^\sharp_1$ is sub-additive, there exists a constant $C_2>0$ such that 
\begin{displaymath} \norm{f*f}_{L^\Phi(G,\omega^\sharp_1)} \le 2C_2 \norm{f}_{L^\Phi(G,\omega^\sharp_1)}\norm{f}_{L^1(G)}. \end{displaymath}
By induction,
\begin{displaymath} \norm{f^{*2n}}_{L^\Phi(G,\omega^\sharp_1)} \le (2C_2)^n \norm{f}_{L^\Phi(G,\omega^\sharp_1)} \norm{f}_{L^1(G)}^{2^n-1}.\end{displaymath}
Then, we have that, since $\omega \le \omega^\sharp_1$,
\begin{align*} \nu_{L^\Phi(G,\omega)}(f) &\le \nu_{L^\Phi(G,\omega^\sharp_1)}(f) = \lim_{n \rightarrow \infty} \norm{f^{*{2^n}}}_{L^\Phi(G,\omega^\sharp_1)}^{2^{-n}}\\
& \le \lim_{n \rightarrow \infty} (2C_2)^{n2^{-n}} \norm{f}_{L^\Phi(G,\omega^\sharp_1)}^{2^{-n}} \norm{f}_{L^1(G)}^{1-2^{-n}} \\
&= \norm{f}_{L^1(G)}.
\end{align*}
It then follows that 
\begin{displaymath} \nu_{L^\Phi(G,\omega)}(f) = \lim_{n \rightarrow \infty} \nu_{L^\Phi(G,\omega)}(f^{*n})^{1/n} \le \lim_{n \rightarrow \infty} \norm{f^{*n}}_{L^1(G)}^{1/n} = \nu_{L^1(G)}(f).  \end{displaymath}
This completes the proof. \qed

We now use Theorem \ref{thm:1} to prove Theorem \ref{thm:2} from the introduction. The proof of Theorem \ref{thm:2} relies heavily on \cite[Lemma 3.1]{FGL06}. We state this lemma below for ease of reference, though we have slightly rephrased the statement for our purposes.

\begin{lem}\cite[Lemma 3.1]{FGL06}\label{lem:spec}
Let $\mathcal{A} \subseteq \mathcal{B}$ be a nested pair of Banach $*$-algebras that either both have a common identity element or both have no identity. Then the following are equivalent:
\begin{enumerate}
   \item[(i)] $\partial \sigma_{\mathcal{A}}(a) \subseteq \partial\sigma_{\mathcal{B}}(a)$ for all $a \in \mathcal{A}$;
   \item[(ii)] $\partial \sigma_{\mathcal{A}}(a) \subseteq \sigma_{\mathcal{B}}(a)$ for all $a \in \mathcal{A}$;
   \item[(iii)] $\nu_{\mathcal{A}}(a) = \nu_{\mathcal{B}}(a)$ for all $a \in \mathcal{A}$.
\end{enumerate}
Furthermore, if $\mathcal{B}$ is Hermitian, then $(i)$, $(ii)$ and $(iii)$ are equivalent to
\begin{enumerate}
   \item[(iv)] $\sigma_{\mathcal{A}}(a) = \sigma_{\mathcal{B}}(a)$ for all $a \in \mathcal{A}$.
\end{enumerate}
\end{lem}

We now give the proof of Theorem \ref{thm:2}.\\

\noindent\textit{Proof of Theorem \ref{thm:2}.}
First we prove $(i)$. By \cite[Remark 4.10]{SW20}, $L^1(G)$ is quasi-Hermitian (i.e.\ $\sigma_{L^1(G)}(f) \subseteq \mathbb{R}$ for every self-adjoint $f \in C_c(G)$) since every compactly-generated subgroup of $G$ is compact. Then, for a self-adjoint $f \in C_c(G) \subseteq L^\Phi(G,\omega)$, by Lemma \ref{lem:spec} and Theorem \ref{thm:1}, we have that $\partial\sigma_{L^\Phi(G,\omega)}(f) \subseteq \sigma_{L^1(G)}(f) \subseteq \mathbb{R}$. But this implies that $\sigma_{L^\Phi(G,\omega)}(f) \subseteq \mathbb{R}$ which completes the proof. Part $(ii)$ is just a consequence of Proposition \ref{prop:Orc2}. \qed

Since every nilpotent group is Hermitian, it follows from Theorem \ref{thm:2} that every weighted Orlicz $*$-algebra on a nilpotent locally elliptic group with respect to a sub-additive weight is Hermitian. It should be noted, however, that not every locally elliptic group is Hermitian. Indeed, there exists the following result.

\begin{thm}\cite{Hul80}
There exists a class 2 solvable locally finite group $\Gamma$ such that $\ell^1(\Gamma)$ is not Hermitian.
\end{thm}

We give the definition of two different non-Hermitian locally finite groups in Section \ref{sec:8} of this article, including the group from \cite{Hul80}.

If $\Gamma$ is the group in the theorem and $A$ some non-discrete abelian locally elliptic group (e.g. $\mathbb{Q}_p$), then $A \times \Gamma$ is a non-discrete class 2 solvable locally elliptic group, and it is not Hermitian since $\Gamma$ is a non-Hermitian quotient of $A \times \Gamma$. In particular, we have the following corollary.

\begin{cor}
There exist (infinitely many) non-discrete class 2 solvable locally elliptic groups which are not Hermitian.
\end{cor}

Any group of the form $A \times \Gamma$ has non-trivial torsion elements since $\Gamma$ is locally-finite. Understanding what happens in the case of torsion-free solvable locally elliptic groups would be an interesting pursuit. We pose this as an open question.

\begin{que}
Does there exist a torsion-free class 2 solvable locally elliptic group which is not Hermitian?
\end{que}

Also, it would be interesting to have characterisations of Hermitianness for $L^1(G,\omega)$ which are independent of the Hermitian property of $L^1(G)$, when $G$ is a locally elliptic group. Although we cannot say anything in general at the moment, we note the following result and an immediate corollary of it.

\begin{thm}\label{thm:Pyt1}\cite[Theorem 1]{Pyt82}
Let $G$ be a locally compact group and $\omega$ a sub-additive weight on $G$. Suppose that there exists $p \in (0,\infty)$ such that $\omega^{-1} \in L^p(G)$. Then $L^1(G,\omega)$ is Hermitian.
\end{thm}

\begin{cor}
Let $G = \bigcup_{i=1}^\infty K_i$ be a locally elliptic group and let $\mu$ denote a left Haar measure on $G$. Suppose that $\bold{a} := (a_i)_{i=1}^\infty \subset \mathbb{R}_{\ge 0}$ is a non-decreasing sequence such that there exists $p \in (0,\infty)$ with
\begin{displaymath} \frac{\mu(K_1)}{a_1^p} + \sum_{i=2}^\infty \frac{\mu(K_i \setminus K_{i-1})}{a_i^p} < \infty. \end{displaymath}
Then $L^1(G,\omega_\bold{a})$ is Hermitian.
\end{cor}

\begin{proof}
Indeed, since
\begin{displaymath} \omega_{\bf{a}}^{-1} = \frac{\chi_{K_1}}{a_1} + \sum_{i=2}^\infty  \frac{\chi_{K_i \setminus K_{i-1}}}{a_i} \end{displaymath}
and
\begin{align*} &\norm{\omega_{\bf{a}}^{-1}}_{L^p(G)}^p  = \int_G  \lvert \omega_{\bf{a}}^{-1}(x) \rvert^p \: d\mu(x) \\
&= \int_{K_1} \lvert \omega_{\bf{a}}^{-1}(x) \rvert^p \: d\mu(x) + \sum_{i=2}^\infty \int_{K_i \setminus K_{i-1}} \lvert \omega_{\bf{a}}^{-1}(x) \rvert^p \: d\mu(x) \\    
&= \frac{\mu(K_1)}{a_1^p} + \sum_{i=2}^\infty \frac{\mu(K_i \setminus K_{i-1})}{a_i^p}, \\
\end{align*}
the result follows directly from these calculations and Theorem \ref{thm:Pyt1}.
\end{proof}

We now note some other assumptions that imply Hermitianness for weighted Orlicz (or $L^1$) $*$-algebras on locally elliptic groups that do not require the weight to be sub-additive. First, we give the following definition.

\begin{dfn}
Let $G$ be a locally elliptic group and $\omega$ a weight on $G$. Suppose that there exist compact open subgroups $K_i \le G$ ($i \in \mathbb{N} \cup \{0\}$) such that $K_i \le K_{i+1}$ for each $i$, $G = \bigcup_{i=0}^\infty K_i$, and the weight
\begin{displaymath} \omega'(n) := \max_{x \in K_{\lvert n \rvert}} \omega(x) \; \; \; (n \in \mathbb{Z})  \end{displaymath}
on $\mathbb{Z}$ satisfies the GRS condition. Then we say that $\omega$ satisfies the \hlight{uniform GRS condition}.
\end{dfn}

We now prove the following lemma. The proof is basically identical to \cite[Theorem 3.4]{FGL06} with some small modifications.

\begin{lem}\label{lem:unigrs}
Let $G$ be a locally elliptic group and $\omega$ a weight on $G$ that satisfies the uniform GRS condition. Then, for all $f \in L^1(G,\omega)$, $\sigma_{L^1(G,\omega)}(f) = \sigma_{L^1(G)}(f)$.
\end{lem}

\begin{proof}
Let $G$ be a locally elliptic group and $\omega$ a weight on $G$ satisfying the uniform GRS condition. Let $K_i \le G$ ($i \in \mathbb{N} \cup \{0\}$) be compact open subgroups of $G$ such that $K_i \le K_{i+1}$ for each $i$, $G = \bigcup_{i=0}^\infty K_i$ and the weight on $\mathbb{Z}$ defined by
\begin{displaymath} \omega'(n) := \max_{x \in K_{\lvert n \rvert}} \omega(x) \; \; \; (n \in \mathbb{Z})  \end{displaymath}
has the GRS property.
Then, for $f \in L^1(G,\omega)$, one computes by induction that
\begin{align*} &\norm{f^{*n}}_{L^1(G,\omega)} \le \int_G \cdots \int_G \lvert f(x_1) \rvert \cdots \lvert f(x_n) \rvert \omega(x_1x_2\cdots x_n) dx_1dx_2\cdots dx_n  \\
& = \sum_{i_1,\dots,i_n=0}^\infty \int_{K_{i_1}\setminus K_{i_1-1}}  \cdots \int_{K_{i_n}\setminus K_{i_n-1}} \lvert f(x_1) \rvert  \cdots \lvert f(x_n) \rvert \omega(x_1x_2\cdots x_n) dx_1dx_2\cdots dx_n
\end{align*}
where we use the convention that $K_{-1}$ is the empty set.

Suppose that $x_j \in K_{i_j} \setminus K_{i_j-1}$ for $j =1,2,\dots,n$. Let $i_k$ be the largest index of $i_1, \dots, i_n$. Then, since $x_1 \cdots x_n \in K_{i_k}$,
\begin{displaymath} \omega(x_1\cdots x_n) \le \max_{x \in K_{i_k}} \omega(x) =\omega'(i_k) \le \omega'(i_1+ \cdots + i_n).\end{displaymath}
Then, let $a_i := \int_{K_i \setminus K_{i-1}} \lvert f(x) \rvert \:dx$ for each $i \in \mathbb{N} \cup \{0\}$ and $a := (a_n)_{n \in \mathbb{N}\cup\{0\}}$. Clearly $\norm{f}_{L^1(G)} = \norm{a}_{\ell^1(\mathbb{Z})}$.
It then follows by the previous arguments that
\begin{displaymath} \norm{f^{*n}}_{L^1(G,\omega)} \le \sum_{i_1,\dots,i_n=1}^\infty a_{i_1}\cdots a_{i_n} \omega(i_1 + \cdots + i_n) = \norm{a^{*n}}_{\ell^1(\mathbb{Z},\omega')}. \end{displaymath}
Hence, one computes that
\begin{align*} \nu_{L^1(G,\omega)}(f) = \lim_{n \rightarrow \infty} &\norm{f^{*n}}^{1/n}_{L^1(G,\omega)} \le \lim_{n \rightarrow \infty} \norm{a^{*n}}^{1/n} = \nu_{\ell^1(\mathbb{Z},\omega')}(a) \\ &= \nu_{\ell^1(\mathbb{Z})}(a) = \norm{a}_{\ell^1(\mathbb{Z})} = \norm{f}_{L^1(G)} \end{align*}
where the third and fourth equalities follow from \cite[Lemma 3.3]{FGL06}. Then 
\begin{displaymath} \nu_{L^1(G,\omega)}(f) = \lim_{n \rightarrow \infty}\nu_{L^1(G,\omega)}(f^n)^{1/n} \le \lim_{n \rightarrow \infty} \norm{f^{*n}}_{L^1(G)} = \nu_{L^1(G)}(f) \end{displaymath}
from which it follows that $\nu_{L^1(G,\omega)}(f) = \nu_{L^1(G)}(f)$. The result then follows by Lemma \ref{lem:spec}.
\end{proof}

We then have the following theorem.

\begin{thm}
Let $G$ be a Hermitian locally elliptic group and $(\Phi,\Psi)$ a complementary pair of Young functions. The following hold:
\begin{enumerate}[(i)]
   \item If $\omega$ is an $L^\Phi$-weight on $G$ with $1/\omega \in L^\Psi(G)$ and $\omega$ has bounded variation, then $L^\Phi(G,\omega)$ is Hermitian; \
   \item If $\omega$ is an arbitrary weight on $G$ satisfying the uniform GRS condition, then $L^1(G,\omega)$ is Hermitian.
\end{enumerate}
\end{thm}

\begin{proof}
Part $(i)$ follows directly from Proposition \ref{prop:bvweight} and Theorem \ref{thm:2}. Part $(ii)$ is a direct consequence of Lemma \ref{lem:unigrs}.
\end{proof}


\section{Functional calculus and the Wiener property}\label{sec:5}

In this section we prove Theorem \ref{thm:3}. The proof requires us to develop a certain (smooth) functional calculus on a total part of a weighted Orlicz $*$-algebra on a locally elliptic group. Our arguments model those used in Section 4 and Section 5 of \cite{FGLLMB03}.

Throughout this section, unless otherwise specified, we assume the hypotheses of Theorem \ref{thm:3}. In particular, $G$ will be a locally elliptic group, $(\Phi,\Psi)$ a pair of complementary Young functions and $\omega$ an $L^\Phi$-weight on $G$ with $1/\omega \in L^\Psi(G)$.

Given $f \in L^\Phi(G,\omega)$, we define the function
\begin{displaymath}  u(f) := \sum_{k=1}^\infty \frac{f^{*k}}{k!}. \end{displaymath}

Since 
\begin{displaymath} \norm{u(f)}_{L^\Phi(G,\omega)} \le \sum_{k=1}^\infty \frac{\norm{f}_{L^\Phi(G,\omega)}^k}{k!} = e^{\norm{f}_{L^\Phi(G,\omega)}} -1 \end{displaymath}
it follows that $u(f) \in L^\Phi(G,\omega)$ for all $f \in L^\Phi(G,\omega)$.

The following result of Pytlik is critical to our results.

\begin{lem}\label{lem:Pyt2}\cite[Lemma 4]{Pyt82}
Let $G$ be a locally compact group and $\omega$ a sub-additive weight on $G$. Suppose that there exists $p \in (0,\infty)$ such that $\omega^{-1} \in L^{p}(G)$. Then, for any self-adjoint $f \in L^1(G, \omega) \cap L^2(G)$ and $\gamma > \log_2((2p+2)/(p+2))$,
\begin{displaymath} \norm{u(inf)}_{L^1(G,\omega)} = O(e^{n^{\gamma}}) \text{ as $n \rightarrow \infty$}. \end{displaymath} 
\end{lem}

We will now generalise this to the algebra $L^\Phi(G,\omega)$ on a locally elliptic group. To do this, we use the following result which is proved on page 142 of \cite{KMB12}. We include the proof due to its simplicity.

\begin{lem}\cite[pg.\ 142]{KMB12}\label{lem:lpineq}
For all $f \in L^\Phi(G,\omega)$ and $n \in \mathbb{N}$,
\begin{displaymath} u(nf) = nu(f) + \sum_{k=1}^{n-1}u(kf)*u(f).  \end{displaymath}
\end{lem}

\begin{proof}
Note that, as a formal expression, $u(f) = e^{f}-1$. Then,
\begin{align*}
  u(nf) &= e^{nf}-1 = e^{(n-1)f}*e^{f}-1 = (u((n-1)f)+1)*(u(f)+1)-1\\
  & = u((n-1)f)*u(f)+u((n-1)f)+u(f).
\end{align*}
A simple induction argument then gives the result.
\end{proof}

As a consequence, the following result holds.

\begin{lem}\label{lem:lphibound}
Let $G$ be a locally elliptic group, $(\Phi, \Psi)$ a complementary pair of Young functions and $\omega$ an $L^\Phi$-weight on $G$. Let $(K_i)_{i \in \mathbb{N}}$ be a standard decomposition of $G$ and $\omega^\sharp_1$ the weight as defined in Proposition \ref{prop:polyw1} with respect to the compact open subgroups $(K_i)_{i \in \mathbb{N}}$. Then, for any self-adjoint $f \in L^\Phi(G,\omega) \cap L^1(G,\omega^\sharp_1) \cap L^2(G)$, and $\gamma > \log_2(4/3)$, 
\begin{displaymath} \norm{u(inf)}_{L^\Phi(G,\omega)} = O(e^{2n^\gamma}) \text{ as $n \rightarrow \infty$}. \end{displaymath}
\end{lem}

\begin{proof}
Throughout the proof, we fix a self-adjoint $f \in L^\Phi(G,\omega) \cap L^1(G,\omega^\sharp_1) \cap L^2(G)$ and $\gamma > \log_2(4/3)$. Since the weight $\omega^\sharp_1$ is a sub-additive weight on $G$ and $1/\omega^\sharp_1\in L^1(G)$, it follows by Lemma \ref{lem:Pyt2} that there exists a constant $C>0$, such that for large $n$,
\begin{displaymath} \norm{u(inf)}_{L^1(G,\omega^\sharp_1)} \le C e^{n^\gamma}.  \end{displaymath} 

Then, by applying Lemma \ref{lem:lpineq} and Proposition \ref{prop:Orc3}(iii), we have that
\begin{align*} 
\norm{u(inf)}_{L^\Phi(G,\omega)} &\le n \norm{u(if)}_{L^\Phi(G,\omega)} + \sum_{k=1}^{n-1} \norm{u(ikf)}_{L^1(G,\omega)}\norm{u(if)}_{L^\Phi(G,\omega)} \\
&\le n \norm{u(if)}_{L^\Phi(G,\omega)} + \sum_{k=1}^{n-1} \norm{u(ikf)}_{L^1(G,\omega^\sharp_1)}\norm{u(if)}_{L^\Phi(G,\omega)}.  
\end{align*}

It follows that there exist constants $C'$ and $C''$ such that
\begin{align*}
\norm{u(inf)}_{L^\Phi(G,\omega)} &\le C' (n  + \sum_{k=1}^{n-1} \norm{u(ikf)}_{L^1(G,\omega^\sharp_1)}) \\
&\le C' (n+ (n-1) \norm{u(i(n-1)f)}_{L^1(G,\omega^\sharp_1)}) \\
&\le C''n e^{n^\gamma} \le C''e^{2n^\gamma}
\end{align*}
for $n$ large. This is precisely what we needed to show.
\end{proof}

We are now able to construct our functional calculus. To do this, we define the following algebra of functions.

\begin{dfn}
Given $0 < \gamma < 1$, define a weight on $\mathbb{Z}$ by $\omega_\gamma(n) := e^{2\lvert n \rvert^\gamma}$ ($n \in \mathbb{Z}$). Let $\tilde{A}_\gamma$ denote the algebra of $2\pi$-periodic $C^\infty$-functions $\varphi: \mathbb{R} \rightarrow \mathbb{R}$ with Fourier coefficients in $\ell^1(\mathbb{Z},\omega_\gamma)$ under pointwise multiplication. In particular, the algebra $\tilde{A}_\gamma$ consists of functions of the form $\varphi(x) = \sum_{n \in \mathbb{Z}} \hat{\varphi}(n) e^{inx}$ such that $\sum_{n \in \mathbb{Z}} \lvert \hat{\varphi}(n) \rvert e^{2\lvert n \rvert^\gamma} < \infty$. Then, we let $A_\gamma$ be the subalgebra of $\tilde{A}_\gamma$ consisting of those functions $\varphi \in \tilde{A}_\gamma$ with $\varphi(0)=0$.
\end{dfn}

Before constructing our functional calculus, we will state the following fact about the algebra $A_\gamma$. This is a consequence of Lemma 1.24 and Theorem 2.11 in \cite{Dom56}.

\begin{prop}\cite{Dom56}\label{prop:Agamma}
For $0<\gamma<1$, the algebra $A_\gamma$ contains functions of arbitrarily small support. Furthermore, for every $\epsilon >0$ and every interval $[p,q] \subseteq (0,2\pi)$ with $p+\epsilon < q - \epsilon$, there exists a function $\varphi \in A_\gamma$ satisfying the following properties:
\begin{enumerate}[(i)]
   \item $0 \le \varphi \le 1$; \
   \item $\supp(\varphi) \cap [0,2\pi] \subseteq [p,q]$; \
   \item $\varphi(x) =1$ for all $x \in [p+\epsilon,q-\epsilon]$. 
\end{enumerate}
\end{prop}

We now prove the following proposition which gives the required functional calculus.

\begin{prop}\label{prop:lphicalc}
Let $G$ be a locally elliptic group, $(\Phi, \Psi)$ a complementary pair of Young functions and $\omega$ an $L^\Phi$-weight on $G$. Fix a $\gamma$ with $\log_2(4/3) < \gamma < 1$. Then, for any $\varphi \in A_\gamma$ and self-adjoint $f \in L^\Phi(G,\omega) \cap L^1(G,\omega^\sharp_1) \cap L^2(G)$,
\begin{displaymath} \varphi\{f\} := \sum_{n\in \mathbb{Z}}\hat\varphi(n) u(inf)  \end{displaymath}
has finite $L^\Phi(G,\omega)$-norm and hence converges to an element of $L^\Phi(G,\omega)$.
\end{prop}

\begin{proof}
Indeed, by Lemma \ref{lem:lphibound} and the definition of $A_\gamma$, we have that
\begin{align*}
\norm{\varphi\{f\}}_{L^\Phi(G,\omega)} & \le \sum_{n \in \mathbb{Z}}\lvert \hat\varphi(n) \rvert \norm{u(inf)}_{L^\Phi(G,\omega)} \\
&\le C \sum_{n\in \mathbb{Z}} \lvert \hat\varphi(n) \rvert  e^{2\lvert n \rvert^\gamma} < \infty \\ 
\end{align*}
for some fixed constant $C>0$. The result follows immediately.
\end{proof}

In particular, by the proposition, the algebra $A_\gamma$ acts on compactly supported continuous self-adjoint functions in $L^\Phi(G,\omega)$.

Since $A_\gamma$ is an algebra under pointwise multiplication, one can check, as done in \cite[Section 4.7]{FGLLMB03}, that for $\varphi,\psi \in A_\gamma$ and a self-adjoint $f \in L^\Phi(G,\omega) \cap L^1(G,\omega^\sharp_1) \cap L^2(G)$,
\begin{displaymath} (\varphi \cdot \psi)\{f\} = \varphi\{f\} * \psi\{f\} \end{displaymath}
and for any $*$-representation $\pi$ of $L^\Phi(G,\omega)$, 
\begin{displaymath} \pi(\varphi\{f\}) = \varphi(\pi(f)), \end{displaymath}
where $\varphi(\pi(f))$ denotes the usual functional calculus of the bounded operator $\pi(f)$.

Now, we are going to work towards the proof of Theorem \ref{thm:3}. First, we note that if $G$ is locally elliptic and $\omega$ a weight on $G$, then, given a fixed compact neighbourhood $K$ of the identity in $G$, there exists a bounded left approximate identity $(f_j)_{j \in J} \subseteq L^1(G,\omega)$ such that $f_j = f^*_j$ and $\supp(f_j) \subseteq K$ for each $j \in J$. Such an approximate identity exists by \cite[Theorem 4.1]{Kuz08}.

We now prove the following result. The proof uses arguments given in \cite[Section 5]{FGLLMB03}.

\begin{lem}\label{lem:conl1}
Let $G$ be a locally elliptic group, $\omega$ a weight on $G$ and $K \subseteq G$ a compact neighbourhood of the identity. Let $(f_j)_{j \in J} \subseteq L^1(G,\omega)$ be a bounded left approximate identity such that each $f_j$ is self-adjoint and $\supp(f_j) \subseteq K$ for all $j$. Then, for any $\gamma$ with $\log_2(4/3) < \gamma < 1$, there exists a function $\varphi \in A_\gamma$ such that, for all $g \in C_c(G)$, $\norm{\varphi\{f_j\}*g - g}_{L^1(G,\omega)} \rightarrow 0.$
\end{lem}

\begin{proof}
Fix $\gamma$ with $\log_2(4/3) < \gamma < 1$. By the arguments of \cite[Section 5]{FGLLMB03}, it suffices to show that we can choose $\varphi \in A_\gamma$ with $\varphi(1)=1$, so that for any $\epsilon > 0$, there exists $N \in \mathbb{N}$ such that
\begin{displaymath} \sum_{\lvert n \rvert > N} \lvert \hat\varphi(n) \rvert \norm{e^{inf_j}*g}_{L^1(G,\omega)} < \epsilon. \end{displaymath}
Now, let $(K_i)_{i \in \mathbb{N}}$ be a standard decomposition of $G$ and $\omega^\sharp_1$ the weight as defined in Proposition \ref{prop:polyw1} with respect to the compact open subgroups $(K_i)_{i \in \mathbb{N}}$. By Proposition \ref{prop:Orc2}(i), there exists a constant $C >0$ such that
\begin{align*} &\norm{e^{inf_j}*g}_{L^1(G,\omega)} \le  \norm{e^{inf_j}*g}_{L^1(G,\omega_1^\sharp)} \\
&\le \norm{g}_{L^1(G,\omega_1^\sharp)} + \norm{u(inf_j)*g}_{L^1(G,\omega_1^\sharp)} \\
&\le \norm{g}_{L^1(G,\omega_1^\sharp)} + C(\norm{u(inf_j)}_{L^1(G,\omega_1^\sharp)} \norm{g}_{L^1(G)} + \norm{g}_{L^1(G,\omega_1^\sharp)} \norm{u(inf_j)}_{L^1(G)}) \\
&\le \norm{g}_{L^1(G,\omega_1^\sharp)} + C(\norm{g}_{L^1(G)} + \norm{g}_{L^1(G,\omega_1^\sharp)})\norm{u(inf_j)}_{L^1(G,\omega_1^\sharp)}.
\end{align*}
Since $\norm{u(inf_j)}_{L^1(G,\omega_1^\sharp)} = O(e^{n^\gamma}) \text{ as $n \rightarrow \infty$}$ by Lemma \ref{lem:Pyt2}, and all other expressions in the last line are constants in the variable $n$, it follows that
\begin{displaymath} \norm{e^{inf_j}*g}_{L^1(G,\omega)} = O(e^{n^\gamma}) \text{ as $n \rightarrow \infty$}. \end{displaymath}
Then, by Proposition \ref{prop:Agamma}, we can choose a $\varphi \in A_\gamma$ such that
\begin{enumerate}[(i)]
   \item $\varphi = 0$ in a neighbourhood of 0;\
   \item $\varphi(1) = 1$;\
   \item $\supp(\varphi) \cap [0,2\pi]$ is compact.
\end{enumerate}
For this $\varphi$ we have that, for some constant $C>0$,
\begin{displaymath} \sum_{n \in \mathbb{Z}} \lvert \hat\varphi(n) \rvert \norm{e^{inf_j}*g}_{L^1(G,\omega)} \le C \sum_{n \in \mathbb{Z}} \lvert \hat\varphi(n) \rvert e^{\lvert n \rvert^\gamma} < \infty. \end{displaymath}
Thus, for any $\epsilon > 0$, we can find $N \in \mathbb{N}$ such that:
\begin{displaymath} \sum_{\lvert n \rvert > N} \lvert \hat\varphi(n) \rvert \norm{e^{inf_j}*g}_{L^1(G,\omega)} < \epsilon. \end{displaymath}
This $\varphi$ then satisfies the lemma.
\end{proof}

A similar result can be proved for general weighted Orlicz $*$-algebras.

\begin{lem}\label{lem:conlp}
Let $G$ be a locally elliptic group, $(\Phi, \Psi)$ a complementary pair of Young functions, $\omega$ an $L^\Phi$-weight on $G$ and $K \subseteq G$ a compact neighbourhood of the identity. Let $(f_j)_{j \in J} \subseteq L^1(G,\omega)$ be a bounded left approximate identity such that each $f_j$ is self-adjoint and $\supp(f_j) \subseteq K$ for all $j$. Then, for any $\gamma$ with $\log_2(4/3) < \gamma < 1$, there exists a function $\varphi \in A_\gamma$ such that, for all $f,g \in C_c(G)$, 
\begin{displaymath} \norm{\varphi\{f_j\}*f*g - f*g}_{L^\Phi(G,\omega)} \rightarrow 0. \end{displaymath}
\end{lem}

\begin{proof}
Fix a $\gamma$ with $\log_2(4/3) < \gamma < 1$. By Lemma \ref{lem:conl1}, we can find a function $\varphi \in A_\gamma$ such that $\norm{\varphi\{f_j\}*g - g}_{L^1(G,\omega)} \rightarrow 0$ for all $g \in C_c(G)$. Then, let $f,g \in C_c(G)$. Since
\begin{displaymath} \norm{\varphi\{f_j\}*f*g - f*g}_{L^\Phi(G,\omega)} \le \norm{\varphi\{f_j\}*f - f}_{L^1(G,\omega)} \norm{g}_{L^\Phi(G,\omega)} \end{displaymath}
by Proposition \ref{prop:Orc2}(i), and since 
\begin{displaymath} \norm{\varphi\{f_j\}*f - f}_{L^1(G,\omega)} \rightarrow 0, \end{displaymath}
it follows that
\begin{displaymath} \norm{\varphi\{f_j\}*f*g - f*g}_{L^\Phi(G,\omega)} \rightarrow 0. \qedhere\end{displaymath}
\end{proof}

%

We now prove Theorem \ref{thm:3}. The argument already exists in literature (see \cite[Section 3, (10)]{Lep76} and \cite[Theorem 6.3]{KMB12}), but we provide it here for completeness.\\

\noindent\textit{Proof of Theorem \ref{thm:3}.}

For the proof we fix $\log_2(4/3) < \gamma < 1$. Let $I \subseteq L^\Phi(G,\omega)$ be a proper non-trivial closed two-sided ideal. Let $(f_j)_{j \in J} \subseteq L^1(G,\omega)$ and $\varphi \in A_\gamma$ be defined as in Lemma \ref{lem:conlp}. 

Now, we claim that $\varphi\{f_j\} \notin I$ for at least one $j$. Indeed, suppose the contrary. Then, $\varphi\{f_j\} *f*g \in I$ for all $j \in J$ and $f,g \in C_c(G)$ since $I$ is an ideal. By Lemma \ref{lem:conlp} and the fact that $I$ is closed, it follows that $f*g \in I$ for all $f,g \in C_c(G)$. This implies that $I = L^\Phi(G,\omega)$ by density of $C_c(G)$. This contradicts the fact that $I$ is assumed to be a proper ideal.

Thus, we may now fix $j \in J$ such that $\varphi\{f_j\} \notin I$. Let $\psi \in A_\gamma$ be a function which is identically 1 on the support of $\varphi$, which exists by Proposition \ref{prop:Agamma}. Then, $\psi\{f_j\} * \varphi\{f_j\} = (\psi\varphi)\{f_j\} = \varphi\{f_j\}$. In particular, the element $\psi\{f_j\} * \varphi\{f_j\} = \varphi\{f_j\}$ is a non-trivial element of the quotient $L^\Phi(G,\omega)/I$, and it follows that the image of $\psi(f_j)$ in the quotient $L^\Phi(G,\omega)/I$ contains 1 in its spectrum. Thus, $L^\Phi(G,\omega)/I$ is not a radical Banach algebra, and hence there exists a non-trivial algebraically irreducible representation of $L^\Phi(G,\omega)/I$ \cite[Theorem 2.3.3]{Pal94}. Lifting this representation to a representation of $L^\Phi(G,\omega)$ then gives a non-trivial algebraically irreducible representation of $L^\Phi(G,\omega)$ that annihilates $I$. This completes the proof.
\qed \\

\section{The $*$-regular property and representation theory of locally elliptic groups}\label{sec:6}

Let $G$ be a locally elliptic group, $(\Phi,\Psi)$ a complementary pair of Young functions with $\Phi \in \Delta_2$, and $\omega$ an $L^\Phi$-weight such that $1/\omega \in L^\Psi(G)$. As mentioned in the introduction, since $\Phi \in \Delta_2$, the (vector space) dual of $L^\Phi(G,\omega)$ is $L^\Phi(G,\omega)^* = L^\Psi(G,\omega^{-1})$ and it follows that every linear functional $\lambda \in L^\Phi(G,\omega)^*$ is of the form 
\begin{displaymath} \lambda(f) = \int_G f(x)g(x) \: dx \; \; \; (f \in L^\Phi(G,\omega)) \end{displaymath}
for some $g \in L^\Psi(G,\omega^{-1})$. Then, one checks that the arguments in \cite[Section 2.4]{KMB12} work identically for the algebra $L^\Phi(G,\omega)$, in particular, the map
\begin{displaymath} \pi \mapsto \bigg( f \mapsto \int_G f(x) \pi(x) \: dx \bigg) \end{displaymath}
provides a bijection between unitary representations of $G$ and unitary representations (i.e.\ non-degenerate $*$-representations) of $L^\Phi(G,\omega)$. Furthermore, this map preserves irreducibility and equivalence of representations. As a consequence of this, the space $\widehat{L^\Phi(G,\omega)}$ can be identified with $\widehat{G}$. This also implies that $C^*(L^\Phi(G,\omega)) = C^*(G)$ since $C_c(G) \subseteq L^\Phi(G,\omega) \subseteq L^1(G)$; this is precisely Theorem \ref{thm:4}(i).

In the following, given a unitary representation $\pi$ of $G$ and $f$ an element of $L^\Phi(G,\omega)$ (resp.\ $L^1(G)$), we will use the convention that $\pi$ also denotes the $*$-representation of $L^\Phi(G,\omega)$ (resp. $L^1(G)$) defined on $f$ by
\begin{displaymath} \pi(f) := \int_G f(x)\pi(x) \: dx. \end{displaymath}
Similarly, $\pi$ will also be used to denote the corresponding unitary representation of the group $C^*$-algebra $C^*(G)$.

We now prove the following result. The statement and proof model that of \cite[Proposition 5.2]{DLMB04}.

\begin{prop}\label{prop:reg}
Let $G$ be a locally elliptic group, $(\Phi, \Psi)$ a complementary pair of Young functions with $\Phi \in \Delta_2$, and $\omega$ an $L^\Phi$-weight on $G$ with $1/\omega \in L^\Psi(G)$. Let $C \subseteq \widehat{G}$ and fix $\rho \in \widehat{G}$. The following are equivalent:
\begin{enumerate}[(i)]
   \item $\cap_{\pi \in C} \ker_{C^*(G)}(\pi) \subseteq \ker_{C^*(G)}(\rho)$; \
   \item $\cap_{\pi \in C} \ker_{L^1(G)}(\pi) \subseteq \ker_{L^1(G)}(\rho)$; \
   \item $\cap_{\pi \in C} \ker_{L^\Phi(G,\omega)}(\pi) \subseteq \ker_{L^\Phi(G,\omega)}(\rho)$; \
   \item $\norm{\rho(f)}_{\B(\H_\rho)} \le \sup_{\pi \in C} \norm{\pi(f)}_{\B(\H_\pi)}$ \text{ for all $f \in C_c(G)$}; \
   \item $\norm{\rho(f)}_{\B(\H_\rho)} \le \sup_{\pi \in C} \norm{\pi(f)}_{\B(\H_\pi)}$ \text{ for all $f \in L^\Phi(G,\omega)$}; \
   \item $\norm{\rho(f)}_{\B(\H_\rho)} \le \sup_{\pi \in C} \norm{\pi(f)}_{\B(\H_\pi)}$ \text{ for all $f \in L^1(G)$}. \
\end{enumerate}
\end{prop}

\begin{proof}
Since $1/\omega \in L^\Psi(G)$, we have that $L^\Phi(G,\omega) \subseteq L^1(G) \subseteq C^*(G)$, and this directly implies that $(i) \implies (ii) \implies (iii)$.

For every group of polynomial growth, and in particular, every locally elliptic group $G$, $L^1(G)$ is $*$-regular \cite[Satz 2]{BLSV78} i.e.\ $\Prim_*(L^1(G))$ is homeomorphic to $\Prim_*(C^*(G))$. Consequently, $(ii)$ implies $(i)$.

Since $C_c(G) \subseteq L^\Phi(G,\omega) \subseteq L^1(G)$, we clearly have that $(vi)\implies(v)\implies(iv)$. On the other hand, since $C_c(G)$ is dense in $L^1(G)$ and $\norm{\sigma(f)}_{\B(\H_\sigma)} \le \norm{f}_{L^1(G)}$ for all $\sigma \in \widehat{G}$ and $f \in L^1(G)$, it follows immediately that $(iv)\implies(vi)$. Thus $(iv)$, $(v)$ and $(vi)$ are equivalent.

Clearly $(vi)$ implies $(ii)$ and hence $(i)$ by $*$-regularity of $L^1(G)$. 

Thus, to complete the proof, we just need to show that $(iii) \implies (iv)$. We follow a similar argument to that used in the proof of \cite[Proposition 5.2]{DLMB04}. Assume, for a contradiction, that $(iv)$ is false i.e.\ that there exists $g \in C_c(G)$ such that
\begin{displaymath} \sup_{\pi \in C} \norm{\pi(g)}_{\B(\H_\pi)} < \norm{\rho(g)}_{\B(\H_\rho)}. \end{displaymath}
We may assume that $g$ is self-adjoint and $\norm{g}_{L^1(G)} \le 1$. Let $\gamma$ be such that $\log_2(4/3) < \gamma < 1$. Recall from Section \ref{sec:5} that the algebra $A_\gamma$ acts by functional calculus on the self-adjoint compactly supported continuous functions in $L^\Phi(G,\omega)$, and in particular, on the function $g$. Now choose $\varphi \in A_\gamma$ that is zero on a neighbourhood of the interval 
\begin{displaymath} [-\sup_{\pi \in C} \norm{\pi(g)}_{\B(\H_\pi)}, \sup_{\pi \in C} \norm{\pi(g)}_{\B(\H_\pi)}] \subseteq [-1,1]\end{displaymath} 
and such that $\varphi(\norm{\rho(g)}_{\B(\H_\rho)}) = 1$. Such a function exists by Proposition \ref{prop:Agamma}. Then, it follows that $\varphi$ is zero on the spectrum of $\pi(g)$ for every $\pi \in C$, hence, $\pi(\varphi\{g\}) = \varphi(\pi(g)) = 0$ for all $\pi \in C$. Also, since $\norm{\rho(g)}_{\B(\H_\rho)}$ is in the spectrum of $\rho(g)$ and $\varphi(\norm{\rho(g)}_{\B(\H_\rho)}) =1$, $\rho(\varphi\{g\}) = \varphi(\rho(g)) \ne 0$. Thus $\varphi\{g\} \notin \ker_{L^\Phi(G,\omega)}(\rho)$ but $\varphi\{g\} \in \cap_{\pi \in C} \ker_{L^\Phi(G,\omega)}(\pi)$. This contradicts $(iii)$ and hence the result follows.
\end{proof}

The proof of Theorem \ref{thm:4}(ii) is a direct consequence of the proposition and the definition of the hull-kernel topology on each of the respective primitive ideal spaces.

In the following proposition, if $\mathcal{A}$ is a Banach $*$-algebra, then $\Max(\mathcal{A})$ will denote the space of all maximal closed two-sided ideals of $\mathcal{A}$ equipped with the hull-kernel topology.

\begin{prop}
Let $G$ be a locally elliptic group, $(\Phi, \Psi)$ a complementary pair of Young functions with $\Phi \in \Delta_2$, and $\omega$ an $L^\Phi$-weight on $G$ with $1/\omega \in L^\Psi(G)$. Suppose that $L^\Phi(G,\omega)$ is Hermitian. Then, the following are equivalent:
\begin{enumerate}[(i)]
   \item $\Prim(C^*(G)) = \Prim_*(C^*(G))$ is $T_1$; \
   \item $\Prim_*(L^1(G))$ is $T_1$; \
   \item $\Prim_*(L^\Phi(G,\omega))$ is $T_1$; \
   \item $\Prim_*(L^\Phi(G,\omega)) \subseteq \Max(L^\Phi(G,\omega))$; \
   \item $\Prim_*(L^\Phi(G,\omega)) = \Max(L^\Phi(G,\omega))$.
\end{enumerate}
\end{prop}

\begin{proof}
The equivalence of $(i)$, $(ii)$ and $(iii)$ follows directly from Theorem \ref{thm:4}. Clearly $(v)$ implies $(iv)$. Since $L^\Phi(G,\omega)$ is Hermitian by assumption, it is Wiener by Theorem \ref{thm:3}, which implies that $\Max(L^\Phi(G,\omega)) \subseteq \Prim_*(L^\Phi(G,\omega))$. Thus $(iv)$ implies $(v)$.

Finally, we need to prove the equivalence of $(iii)$ and $(iv)$. First, let's suppose that $(iii)$ is true and prove $(iv)$. Suppose for a contradiction that there is an ideal $J$ in $\Prim_*(L^\Phi(G,\omega))$ which is not maximal amongst closed two-sided ideals. Then, there exists a proper closed (two-sided) ideal $J' \subseteq L^\Phi(G,\omega)$ properly containing $J$. Since $L^\Phi(G,\omega)$ is Hermitian and hence Wiener by Theorem \ref{thm:3}, $\hull_*(\{J'\})$ is non-empty and hence contains an ideal $I$ in $\Prim_*(L^\Phi(G,\omega))$. In particular, $\hull_*(\{J\}) \supseteq \{I,J\}$, from which it follows, by definition of the hull-kernel topology, that $J$ is not a closed point of $\Prim_*(L^\Phi(G,\omega))$. This contradicts the fact that $\Prim_*(L^\Phi(G,\omega))$ is $T_1$.

We now suppose that $(iv)$ is true and prove $(iii)$. We need to show that under the assumption of $(iv)$, every point in $\Prim_*(L^\Phi(G,\omega))$ is closed. Let $J \in \Prim_*(L^\Phi(G,\omega))$. Then, since $J$ is maximal by the assumption of $(iv)$, it follows that $\hull_*(\ker(\{J\})) = \{J\}$, which implies that $J$ is a closed point in $\Prim_*(L^\Phi(G,\omega))$ by definition of the hull-kernel topology.
\end{proof}

It is an ongoing area of research to determine for which locally elliptic groups $G$ does the topology on $\Prim(C^*(G))$ satisfy the $T_1$ separation axiom. We note that if $G$ is either 2-step-nilpotent, or nilpotent and has a compact open normal subgroup, then it is already known that $\Prim(C^*(G))$ is $T_1$ \cite{Pog83,CM86}.


\section{Minimal ideals of a given hull}\label{sec:7}

Let $G$ be a locally elliptic group, $(\Phi,\Psi)$ a complementary pair of Young functions with $\Phi \in \Delta_2$, and $\omega$ an $L^\Phi$-weight such that $1/\omega \in L^\Psi(G)$. Throughout this section we fix $\gamma$ with $\log_2(4/3) < \gamma < 1$ so that we have a functional calculus of $A_\gamma$ acting on compactly supported self-adjoint functions in $L^\Phi(G,\omega)$. Now, given $C \subseteq \Prim_*(L^\Phi(G,\omega))$ closed, we define the following set, which we view as a subset of $L^\Phi(G,\omega)$:
\begin{align*} m(C) := \{ &\varphi\{f\} : f=f^* \in C_c(G), \norm{f}_{L^1(G)} \le 1, \varphi \in A_\gamma, \varphi = 0 \text{ on a} \\ &\text{neighbourhood of } [-\sup_{\ker(\pi) \in C}\norm{\pi(f)}_{\B(\H_\pi)}, \sup_{\ker(\pi) \in C} \norm{\pi(f)}_{\B(\H_\pi)}] \}.   \end{align*}
Let $j(C)$ denote the closed two-sided ideal of $L^\Phi(G,\omega)$ generated by the set $m(C)$. Then, a similar argument to Lemma 1 and Theorem 1 in \cite{Lud79b} give the following theorem. We repeat the proof for completeness and to clarify the small changes that need to be made to the argument.

\begin{thm}\label{thm:minideal}
Let $G$ be a locally elliptic group, $(\Phi,\Psi)$ a complementary pair of Young functions with $\Phi \in \Delta_2$, and $\omega$ an $L^\Phi$-weight such that $1/\omega \in L^\Psi(G)$. Suppose that $L^\Phi(G,\omega)$ is Hermitian. Then, for every closed set $C \subseteq \Prim_*(L^\Phi(G,\omega))$, the closed two-sided ideal $j(C) \subseteq L^\Phi(G,\omega)$ satisfies the following properties:
\begin{enumerate}[(i)]
   \item $\hull_*(j(C)) = C$; \
   \item For every closed two-sided ideal $I \subseteq L^\Phi(G,\omega)$ with $\hull_*(I) \subseteq C$, $j(C) \subseteq I$.
\end{enumerate}
\end{thm}

\begin{proof}
$(i)$ Let $\ker_{L^\Phi(G,\omega)}(\pi) \in C$ and $\varphi\{f\} \in m(C)$. Then, since $\pi(\varphi\{f\}) = \varphi(\pi(f))$, and $\varphi$ is zero on the spectrum of $\pi(f)$, it follows that $\pi(\varphi\{f\}) = 0$. Thus, $m(C) \subseteq \ker_{L^\Phi(G,\omega)}(\pi)$ for all $\ker_{L^\Phi(G,\omega)}(\pi) \in C$, and it follows that $\hull_*(m(C)) \subseteq C$, which implies that $\hull_*(j(C)) \subseteq C$.

Conversely, suppose that $\ker_{L^\Phi(G,\omega)}(\rho) \notin C$ for some $\rho \in \widehat{L^\Phi(G,\omega)}$. Then, by Proposition \ref{prop:reg}, there exists $f \in C_c(G)$ self-adjoint with $\norm{f}_{L^1(G)} <1$ such that
\begin{displaymath} \sup_{\ker(\pi) \in C} \norm{\pi(f)}_{\B(\H_\pi)} < \norm{\rho(f)}_{\B(\H_\rho)}. \end{displaymath}
Choose $\varphi \in A_\gamma$ such that $\varphi$ is zero in a neighbourhood of the interval
\begin{displaymath} [-\sup_{\ker(\pi) \in C}\norm{\pi(f)}_{B(\H_\pi)}, \sup_{\ker(\pi) \in C}\norm{\pi(f)}_{B(\H_\pi)}]\end{displaymath} 
and $\varphi(\norm{\rho(f)}_{\B(\H_\rho)}) =1$. It then follows that $\rho(\varphi\{f\}) = \varphi(\rho(f)) \ne 0$ and hence $\ker(\rho)$ does not contain $m(C)$. Thus it follows that $\hull_*(j(C)) = C$.

$(ii)$ Let $I \subseteq L^\Phi(G,\omega)$ be a closed two-sided ideal with $\hull_*(I) \subseteq C$. Let $\varphi\{f\} \in m(C)$. Choose non-trivial $\psi \in A_\gamma$ which is identically 1 on the support of $\varphi$, which exists by Proposition \ref{prop:Agamma}. Then, $\psi\{f\} * \varphi\{f\} = \varphi\{f\}$ since $\psi\varphi = \varphi$. Also, it is clear that $\psi\{f\} \in m(C)$ and $\hull_*(\{\psi\{f\}\}) \supseteq \hull_*(m(C)) = C$. Then \cite[Lemma 2]{Lud79b} gives the result.
\end{proof}


\section{Examples and open questions}\label{sec:8}

\subsection{Locally elliptic groups with the bounded index property}

We start with the following definition of a locally elliptic group having the bounded index property.

\begin{dfn}
Let $G$ be a locally elliptic group. The group $G$ has the \hlight{bounded index property} if there exist compact open subgroups $K_i \le G$ for each $i \in \mathbb{N}$ such that $K_i \le K_{i+1}$ for each $i$, $G = \bigcup_{i=1}^\infty K_i$ and $\sup_{n \in \mathbb{N}} [K_{n+1}:K_n] < \infty$.
\end{dfn}

We will give a number of examples of locally elliptic groups with the bounded index property throughout this section. First, we would like to prove the following result, which gives a canonical way of constructing $L^p$-weights, or more generally, $L^\Phi$-weights on these groups for some Young function $\Phi$.

\begin{prop}
Let $G = \bigcup_{i=1}^\infty K_i$ be a locally elliptic group and suppose that $M := \sup_{n \in \mathbb{N}}[K_{n+1}:K_n] < \infty$. Given $f \in \ell^1(\mathbb{N})$ and $q \in [1,\infty)$, define the sub-additive weight
\begin{displaymath} \omega_{f,q} := \chi_{K_1} + \sum_{n=1}^\infty \bigg(\frac{M^{n}}{f(n)}\bigg)^{1/q} \chi_{K_{n+1}\setminus K_n}. \end{displaymath}
Then, $1/\omega_{f,q} \in L^q(G)$. In particular, we have the following:
\begin{enumerate}[(i)]
   \item If $q=1$, then $\omega_{f,1}$ is an $L^\Phi$-weight on $G$ for any Young function $\Phi$;
   \item If $q >1$ and $p \in (1,\infty)$ such that $1/p + 1/q=1$, then $\omega_{f,q}$ is an $L^p$-weight.
\end{enumerate}
\end{prop}

\begin{proof}
The inverse of $\omega_{f,q}$ is given by
\begin{displaymath} 1/\omega_{f,q} = \chi_{K_1} + \sum_{n=1}^\infty \bigg(\frac{f(n)}{M^{n}}\bigg)^{1/q}\chi_{K_{n+1} \setminus K_n}. \end{displaymath}

We now show that $1/\omega_{f,q} \in L^q(G)$. Let $\mu$ be the Haar measure on $G$. We may assume that $\mu$ has been normalised so that $\mu(K_1) = 1$. Then,
\begin{displaymath} \mu(K_2 \setminus K_1) = [K_2:K_1] \mu(K_1) - \mu(K_1) < M \mu(K_1) = M,  \end{displaymath}
and by induction, for $n > 2$,
\begin{displaymath} \mu(K_{n+1} \setminus K_{n}) = [K_{n+1}:K_{n}] \mu(K_{n}) - \mu(K_{n}) < M\mu(K_{n}) \le M^n.  \end{displaymath}
One then computes 
\begin{align*}
\int_G 1/\omega_{f,q}^{q}(x) \: dx &= \int_{K_1} 1/\omega_{f,q}^{q}(x) \: dx + \sum_{n=1}^\infty \int_{K_{n+1} \setminus K_n} 1/\omega_{f,q}^{q}(x) \: dx \\
&< 1 + \sum_{n=1}^\infty f(n) < \infty.
\end{align*}
Thus, $1/\omega_{f,q} \in L^q(G)$. If $q>1$, it follows by Lemma \ref{lem:lpw} that $\omega_{f,q}$ is an $L^p$-weight on $G$ since it is sub-additive, which is $(ii)$. On the other hand, if $q=1$ and $(\Phi,\Psi)$ a complementary pair of Young functions, then $1/\omega_{f,1} \in L^1(G) \cap L^\infty(G) \subseteq L^\Psi(G)$, and since $\omega_{f,1}$ is sub-additive, it follows by Proposition \ref{prop:Orc2} that $L^\Phi(G,\omega_{f,q})$ is a Banach $*$-algebra. Thus $\omega_{f,1}$ is an $L^\Phi$-weight for any Young function $\Phi$. This is $(i)$.
\end{proof}

\subsection{Contraction groups}

We begin with the definition of a contraction group.

\begin{dfn}
Let $G$ be a locally compact group and $\alpha$ a bi-continuous automorphism of $G$. The pair $(G,\alpha)$ is called a \hlight{contraction group} if for all $x \in G$, $\alpha^n(x) \rightarrow \id_G$ as $n \rightarrow \infty$. 
\end{dfn}

We will now give some examples of contraction groups.

\begin{exm}\label{exm:abcont}
\begin{enumerate}[(i)]
   \item The additive group $(\mathbb{Q}_p,+)$ with automorphism multiplication by $p$. This example can be extended to any local field.
   \item Let $F$ be a finite group. The group $(\bigoplus_{\mathbb{Z}_{<0}} F) \times (\prod_{\mathbb{Z}_{\ge 0}} F)$ equipped with the ``right-shift'' automorphism is a torsion contraction group. \
   \item The group $U_n(\mathbb{Q}_p)$ of $n$-dimensional unipotent matrices over $\mathbb{Q}_p$ equipped with the automorphism which is conjugation by the diagonal matrix 
   \begin{displaymath} \text{diag}(1,p,p^2,\dots,p^{n-1}).\end{displaymath}
   The field $\mathbb{Q}_p$ can also be replaced with any local field in this example, provided you change the automorphism appropriately.
\end{enumerate}
\end{exm}

Contraction groups have a fundamental importance in the structure theory of totally disconnected locally compact groups \cite{BW04} and their structure theory has been studied extensively in the papers \cite{GW10,GW21,GW21b}. The following structure theorem about contraction groups is key to their study.

\begin{thm}\cite{GW10,GW21}\label{thm:contgrp}
Let $(G,\alpha)$ be a contraction group. Then, there exists distinct primes $p_1, \dots, p_n$ and unipotent linear algebraic groups $G_{p_1}, \dots, G_{p_n}$ over $\mathbb{Q}_{p_1}, \dots, \mathbb{Q}_{p_n}$ respectively such that
\begin{displaymath} G \cong G_0 \times G_{p_1} \times \dots \times G_{p_n} \times \tor(G)  \end{displaymath}
where $G_0$ denotes the connected component of the identity in $G$ and $\tor(G)$ the subgroup of torsion elements. Furthermore, $G_0$ is a connected simply-connected nilpotent real Lie group.
\end{thm}

If we further assume that the contraction group $G$ is totally disconnected, then we may choose $U \le G$ a compact open subgroup, and it follows from $\alpha$ being contractive that $G = \bigcup_{n=1}^\infty \alpha^{-n}(U)$. Thus, $G$ is locally elliptic since $\alpha^n(U)$ is compact for each $n$, and $G$ has the bounded index property since the indices $[\alpha^{-n-1}(U):\alpha^{-n}(U)]$ are constant with $n$ varying. We state this below.

\begin{prop}
Every totally disconnected locally compact contraction group is locally elliptic and has the bounded index property.
\end{prop}

By Theorem \ref{thm:contgrp}, every torsion-free contraction group is nilpotent and hence it is Hermitian by \cite{Lud79}. However, in contrast, torsion contraction groups need not be nilpotent or even solvable. For example, if $F$ is a non-solvable finite group, then the group
\begin{displaymath}  G:= \bigg( \bigoplus_{\mathbb{Z}_{<0}} F \bigg) \times \bigg( \prod_{\mathbb{Z}_{\ge0}} F \bigg) \end{displaymath}
is a non-solvable torsion contraction group when equipped with the right-shift automorphism as its contractive automorphism. This group, however, is also Hermitian since it is an $[FC]^-$-group \cite[Theorem 2.6]{LM77}. But torsion contraction groups need not be $[FC]^-$-groups either, with the Heisenberg group over $\mathbb{F}_p(\!(t)\!)$ providing an example of a torsion contraction group that is not an $[FC]^-$-group. 

We thus pose the following question, which is an interesting question in the context of the harmonic analysis of totally disconnected locally compact groups.

\begin{que}\label{que:1}
Is every totally disconnected locally compact contraction group Hermitian and hence Wiener? 
\end{que}

Furthermore, it is a non-trivial question as to whether the direct product of two Hermitian groups is Hermitian (\textit{c.f.} \cite{Bon61}). Thus, we also pose the following question which one would need to understand to answer Question \ref{que:1}.

\begin{que}
If $(G,\alpha)$ is a contraction group and $\tor(G)$ is Hermitian/Wiener, is $G$ also Hermitian/Wiener?
\end{que}

The author is pursuing a solution to these questions on contraction groups in ongoing work, and this work ties in with recent research on the representation theory of contraction groups \cite{Car24,CC25}.

\subsection{Unipotent linear algebraic group over non-archimedean local fields}\label{sec:8.3}
 
Let $k$ be a non-archimedean local field and define $U_n(k)$ to be the group of all upper-triangular $n$-dimensional unipotent matrices over $k$. Given $x \in k$ with absolute value $<1$, one checks that conjugation by the diagonal matrix $\text{diag}(1,x,x^2,\dots,x^{n-1})$ is a contractive automorphism of $U_n(k)$. Thus, it follows that $U_n(k)$ is a locally elliptic group with the bounded index property by the discussion in the previous section, and since these properties pass to closed subgroups, we get the following result.

\begin{prop}
Let $N$ be a group of $n$-dimensional unipotent matrices over a non-archimedean field. The group $N$ is locally-elliptic and has the bounded index property.
\end{prop}

Now let's assume that $N$ is a group of $n$-dimensional unipotent matrices over a $p$-adic field for some prime $p$. Let $\frak{n}$ denote the Lie algebra of $N$ and $\frak{n}^*$ the vector space dual of $\frak{n}$. The group $N$ acts on $\frak{n}^*$ by the so called \hlight{codajoint action}
\begin{displaymath}  (\Ad^*(g)\lambda)(X) := \lambda(\Ad(g^{-1})X) \; \; \; (g \in N, X \in \frak{n}, \lambda \in \frak{n}^*) \end{displaymath}
 where $\Ad$ denotes the usual adjoint action of $N$ on $\frak{n}$.
 
 The following result is well known.
 
\begin{thm}\cite{Mor65,BS08}
Let $N$ be a group of $n$-dimensional unipotent matrices over a $p$-adic field. Then the following hold:
\begin{enumerate}[(i)]
   \item $N$ is CCR;
   \item The spaces $\widehat{N}$, $\Prim(C^*(N))$ and $\frak{n}^*/\Ad^*(N)$ are homeomorphic.
\end{enumerate}
\end{thm}

In particular, we have the following corollary by the results in this article.

\begin{cor}
Let $N$ be a group of $n$-dimensional unipotent matrices over a $p$-adic field. Let $(\Phi, \Psi)$ be a complementary pair of Young functions with $\Phi \in \Delta_2$ and $\omega$ an $L^\Phi$-weight on $N$ with $1/\omega \in L^\Psi(N)$. Then, there exists a homeomorphism
\begin{displaymath} \eta: \Prim_*(L^\Phi(N,\omega)) \rightarrow \frak{n}^*/\Ad^*(N).  \end{displaymath}
\end{cor}

Maintaining the notation as in the corollary, let's suppose that $L^\Phi(N,\omega)$ is Hermitian (which is the case if $\omega$ is sub-additive, for example) and suppose that $I$ is a closed ideal of $L^\Phi(N,\omega)$ with $\hull_*(I) := C \subseteq \Prim_*(L^\Phi(N,\omega))$. Let $j(C)$ be the ideal of $L^\Phi(N,\omega)$ as defined in Theorem \ref{thm:minideal}. Then, the ideal $I$ must sit between $\ker(C)$ and $j(C)$ i.e.\ $j(C) \subseteq I \subseteq \ker(C)$. In particular, understanding the closed ideals of $L^\Phi(G,\omega)$ with hull equal to $C$ would require one to understand the ideal theory of the algebra $\ker(C)/j(C)$.

The set $C \subseteq \Prim_*(L^\Phi(N,\omega))$ is called a \hlight{set of synthesis} if $\ker(C)/j(C)$ is trivial, or equivalently, $\ker(C)$ is the only closed ideal of $L^\Phi(N,\omega)$ with hull $C$. The property of $L^\Phi(N,\omega)$ being Wiener is equivalent to the empty set being a set of synthesis, so under our assumption that $L^\Phi(N,\omega)$ is Hermitian, it follows that the empty set is a set of synthesis by Theorem \ref{thm:3}. Also, by the previous corollary, we may identify $C$ with a subset of $\frak{n}^*/\Ad^*(N)$. We thus pose the following problem.

\begin{que}
Let $N$ be a group of $n$-dimensional unipotent matrices over a $p$-adic field. Let $(\Phi, \Psi)$ be a complementary pair of Young functions and $\omega$ an $L^\Phi$-weight on $N$ with $1/\omega \in L^\Psi(N)$ and $L^\Phi(N,\omega)$ Hermitian. Given 
\begin{displaymath} C \subseteq \frak{n}^*/\Ad^*(N) \cong \Prim_*(L^\Phi(N,\omega)), \end{displaymath} 
can geometric properties of $C$, when viewed as a subset of $\frak{n}^*/\Ad^*(N)$, be linked to algebraic properties of the algebra $\ker(C)/j(C)$? In particular, what geometric properties of $C$ imply that it is a set of synthesis?
\end{que}
 
We will now note some results that are known in the case of connected nilpotent Lie groups, in which this question is motivated by. We now suppose that $N_0$ is a connected nilpotent (real) Lie group. The entire discussion of this section regarding Orlicz $*$-algebras on unipotent $p$-adic groups goes through when the (unipotent $p$-adic) group is replaced with the connected nilpotent Lie group $N_0$ and $\omega$ a sub-exponential weight on $N_0$ \cite{OS17}. In particular, $L^1(N_0)$ contains a minimal ideal $j(C)$ for each hull $C \subseteq \Prim_*(L^1(N_0))$.

It is known that for every one-point set $\{J\} \subseteq \Prim_*(L^1(N_0))$, the algebra $\ker(\{J\})/j(\{J\}) = J/j(\{J\})$ is nilpotent \cite{Lud83b}. Furthermore, if $N_0$ is 2-step-nilpotent, then all singleton sets are sets of synthesis, but this does not hold if $N_0$ is 3-step-nilpotent \cite{Lud83a}. In particular, in some sense, if $N_0$ is 2-step-nilpotent, then closed ideals of $L^1(N_0)$ whose hulls are singleton sets are classified. More generally, if $N_0$ is not 2-step-nilpotent, then the set $\{J\}$ is a set of synthesis provided that $\{J\}\subseteq \Prim_*(L^1(N_0))$ corresponds to a ``flat orbit'' in $\frak{n}_0^*/\Ad^*(N_0)$. This was more recently generalised to weighted $L^1$-algebras on connected nilpotent Lie groups in \cite{LMBP13}.

It would be good to try and replicate these results for weighted Orlicz $*$-algebras on unipotent $p$-adic groups, or more generally, torsion-free contraction groups. Even treating the case of the full group algebra $L^1(N)$ for a unipotent $p$-adic group $N$ would be novel.

\subsection{Non-Hermitian locally finite groups}

Here we will write down the definitions of two non-Hermitian locally finite groups found in the papers by Fountain-Ramsey-Williamson \cite{FRW76} and Hulanicki \cite{Hul80}. Expanding on these two examples, it would be an interesting question for future research to characterise the Hermitian locally finite groups. 

\subsubsection{The Fountain-Ramsay-Williamson group}

The Fountain-Ramsay-Williamson group is the locally finite discrete group $G$ with generators $\{x_n : n \in \mathbb{N}\}$ subject to the following relations:
\begin{enumerate}[(i)]
   \item $x^2_n = \id_G$ for all $n \in \mathbb{N}$;
   \item $x_lx_mx_nx_m = x_mx_nx_mx_l$ for all $l,m < n \in \mathbb{N}$.
\end{enumerate}

It is a consequence of \cite[Section 5 \& 6]{FRW76} that the group $G$ is locally finite and not Hermitian. 

\subsubsection{The Leptin-Hulanicki group}
Let $H$ be the direct sum of countably infinite many copies of the cyclic group of order 2. Define the group $N:= \bigoplus_{h \in H} H$. Of course, $H$ acts on $N$ by permuting indices in the direct sum, so we can form the semi-direct product $G := N \rtimes H$. The group $G$ we refer to as the Leptin-Hulanicki group. The group $G$ is defined in \cite[Section 3]{Hul80}, but Hulanicki attributes the example to Leptin, hence we term this group as the \textit{Leptin-Hulanicki group}. This group is class 2 solvable and it is shown in \cite{Hul80} that the group $G$ is not Hermitian.

\subsection{Horocyclic groups of automorphism of trees and scale groups}

Let $d \in \mathbb{N}_{>2}$ and let $\Td = (V\Td,E\Td)$ denote the regular tree of degree $d$, that is, the infinite tree with the property that every vertex has degree $d$. We denote the boundary of $\Td$ by $\partial \Td$. For $\gamma \in \partial \Td$, define two groups by $G_\gamma := \{ x \in \Aut(\Td) : x(\gamma) = \gamma \}$ and  $B_\gamma := \{ x \in G_\gamma : \exists v \in V\Td, x(v) = v \}$. 

\begin{prop}
The group $B_\gamma$ is locally elliptic and $G_\gamma \cong B_\gamma \rtimes \mathbb{Z}$.
\end{prop}

\begin{proof}
This is a standard result. See, for example, \cite{FTN91,Wil20}.
\end{proof}

A closed vertex transitive subgroup $G \le G_\gamma$ is called a \hlight{scale group}. Scale groups are studied in \cite{BW04,Hor15,Wil20} and it is shown they have important connections with the structure theory of totally disconnected locally compact groups. For example, every non-uniscalar totally disconnected locally compact group has a subquotient isomorphic to a scale group. Also, these groups are simultaneously analogues of parabolic subgroups and $ax+b$ groups in the theory of tdlc groups. 

Every scale group is of the form $G = N \rtimes \mathbb{Z}$ with $N$ locally elliptic, so understanding the harmonic analysis of locally elliptic groups is critical to understanding the harmonic analysis of scale groups.


The following is an open question in the harmonic analysis of tdlc groups.

\begin{que}
Let $G = N \rtimes \mathbb{Z}$ be a scale group. Under what assumptions is $G$ Hermitian?
\end{que}

As far as the author is aware, it is even an open question as to whether $\mathbb{Q}_p \rtimes \mathbb{Z}$ is Hermitian (\textit{c.f.} \cite[Section 3.6]{Pal15}).


\begin{center}
\textsc{Acknowledgements}
\end{center}
The author thanks Jared White and Pierre-Emmanuel Caprace for reading a previous version of this manuscript and their feedback/comments.

\bibliographystyle{amsplain}
\bibliography{wei_alg}


\end{document}